\documentclass[12pt]{amsart}

\usepackage{fullpage}
\usepackage{amsmath}
\usepackage{amssymb}
\usepackage{amsthm}
\usepackage{amstext}
\usepackage{yhmath}
\usepackage{appendix}
\usepackage{perpage}
\usepackage{slashed}
\usepackage[x11names]{xcolor}
\usepackage{tikz-cd}
\usepackage{ytableau}\ytableausetup{centertableaux}
\usepackage[colorlinks=true, allcolors=blue]{hyperref}
\usepackage{todonotes}

\newtheorem{thm}{Theorem}[section]
\newtheorem{lemm}[thm]{Lemma}
\newtheorem{coro}[thm]{Corollary}
\newtheorem{prop}[thm]{Proposition}
\newtheorem{conj}[thm]{Conjecture}

\newtheorem*{thmMain}{Theorem}

\theoremstyle{definition}
\newtheorem{expl}[thm]{Example}
\newtheorem{defi}[thm]{Definition}
\newtheorem{remark}[thm]{Remark}

\begin{document}

\title{Jordan induction for $\mathrm{GL}_n(\mathcal{O})$}

\author{Zhe Chen}

\address{Department of Mathematics, Shantou University, Shantou, 515821, China}
\email{zhechencz@gmail.com}

\begin{abstract}
Let $\mathcal{O}$ be a complete discrete valuation ring with a finite residue field. We introduce a simple framework of constructing smooth representations of $\mathrm{GL}_n(\mathcal{O})$ modulo the knowledge of nilpotent orbits, called {Jordan induction}, which enjoys several remarkable properties: It yields only \emph{irreducible} representations, it yields \emph{all} the even level irreducible representations, and up to conjugation it yields \emph{non-isomorphic} representations.

In the even level case, this gives an explicit realisation of Hill's analogue of Lusztig's Jordan decomposition, as well as a vast generalisation of G\'erardin's construction for $\mathrm{GL}_n(\mathcal{O})$. As a simple application, we construct an explicit section to the orbit map. We also propose a conjecture linking Jordan induction and Lusztig induction, aiming at generalising the algebraisation theorem obtained in recent joint works with Stasinski.
\end{abstract}

\maketitle

\tableofcontents

\section{Introduction}

Let $\mathcal{O}$ be a complete discrete valuation ring with finite residue field $\mathbb{F}_q$. The representation theory of $\mathrm{GL}_n(\mathcal{O})$, like that of any group, has two fundamental problems: classifying the irreducible representations and explicitly constructing them. We shall be concerned with the latter one.

\vspace{2mm} Let $\pi$ be a fixed uniformiser  of $\mathcal{O}$. Then every smooth irreducible representation of $\mathrm{GL}_n(\mathcal{O})$ factors through a finite quotient $\mathrm{GL}_n(\mathcal{O}_r)$, where $\mathcal{O}_r:=\mathcal{O}/\pi^r$. 

\vspace{2mm} For $r=1$, we have $\mathrm{GL}_n(\mathcal{O}_r)=\mathrm{GL}_n(\mathbb{F}_q)$, whose representation theory has been extensively studied, and in 1955 Green \cite{Green_1955} obtained all its irreducible characters. In 1977 Lusztig and Srinivasan \cite{Lusztig_Srinivasan_char_finite_unitary_gp_1977} constructed the corresponding representation spaces using Deligne--Lusztig theory \cite{DL1976}. Recently, Jing and Wu \cite{Jing_Wu_2024_J_Alg} initiated an alternative approach based on vertex operator algebras. For general connected reductive groups over $\mathbb{F}_q$, Lusztig established the Jordan decomposition of irreducible characters in \cite{Lusztig_84_OrangeBook,Lusztig_1988_Rep_red_gp_disconnectedcentre}, a cornerstone of their representation theory. For more details and results we refer the interested reader to \cite{Carter1993FiGrLieTy,DM_book_2nd_edition,Geck_Malle_2020book}. 

\vspace{2mm} For $r>1$, the study of representations of $\mathrm{GL}_n(\mathcal{O}_r)$, and of the closely related $\mathrm{SL}_n(\mathcal{O}_r)$, has attracted considerable attention over the past decades: There are cohomological approaches towards generalising Deligne--Lusztig theory \cite{Lusztig2004RepsFinRings,Sta2009Unramified,Sta2011ExtendedDL,Chen_2019_flag_orbit}, algebraic/arithmetic approaches based on Clifford theory and zeta functions \cite{Hill_1995_Regular,AvniKlopschOnnVoll_2016_similarity,Stasinski_Stevens_2016_regularRep,Krakovski_Onn_Singla_regularchar_2018,Hasa_Stasinski_2019_trans_AMS,Hassain_Singla_2022_ADV,OnnPrasadSingla_2025_zetaA2poschar}, and analytic approaches using theta functions/modular forms \cite{Nobs_Wolfart_theta_I_1974,Wolfart_theta_II_1975,Chen_Feng_2025_classnumber_InvariantCharacters}, just to name a few; more developments can be found in the survey \cite{Stasinski_2016_survey}. In the present work we introduce a new framework of constructing representations based on the idea of Jordan decomposition; this framework, on the one hand, is similar to Harish-Chandra theory, and on the other hand has a natural connection with $\ell$-adic cohomology.

\vspace{2mm} Throughout the remaining part of the paper we assume $r>1$. We make no assumption on $p:=\mathrm{char}(\mathbb{F}_q)$.

\vspace{2mm}  In the 1990s, in a series of inspiring works \cite{Hill_1993_Jordan,Hill_1994_nilpotent,Hill_1995_Regular,Hill_1995_semisimple} Hill found an analogue of Lusztig's Jordan decomposition for $\mathrm{GL}_n(\mathcal{O}_r)$ and studied $\mathrm{Irr}(\mathrm{GL}_n(\mathcal{O}_r))$ based on it, with aids from Clifford theory and adjoint orbits. To describe it in a more precise way, for every semisimple $s\in \mathfrak{gl}_n(\mathbb{F}_q) = M_n(\mathbb{F}_q)$, let $\mathrm{Irr}_{s}(\mathrm{GL}_n(\mathcal{O}_r))$ denote the set of equivalence classes of irreducible representations whose orbits (see Definition~\ref{defi:orbit type}) contain an element with $s$ being the semisimple part. Then there is a natural partition 
$$\mathrm{Irr}(\mathrm{GL}_n(\mathcal{O}_r))=\bigsqcup_{(s)}\mathrm{Irr}_{s}(\mathrm{GL}_n(\mathcal{O}_r))$$
where $(s)$ runs over the semisimple conjugacy classes in $M_n(\mathbb{F}_q)$. Hill's version of Jordan decomposition \cite[2.13~Theorem]{Hill_1993_Jordan} asserts that, for each semisimple $s\in M_n(\mathbb{F}_q)$, there exists a bijection
\begin{equation*}
\begin{split}
\mathrm{Irr}_{s}(\mathrm{GL}_n(\mathcal{O}_r))
& \longleftrightarrow 
\{  \textrm{nilpotent orbit representations of } C_{\mathrm{GL}_n(\mathcal{O}_r)}(\hat{s})  \};\\ 
\sigma
& \longmapsto 
\nu_{\sigma},
\end{split}
\end{equation*}
where $\hat{s} \in M_n(\mathcal{O}_r)$ is a lift of $s$, such that $\frac{\dim \sigma}{\dim \nu_{\sigma}}$ is a constant. Replacing ``nilpotent orbit'' by ``unipotent'' this statement becomes an exact analogue of the bijection part of Lusztig's Jordan decomposition; see \cite[Theorem~11.5.1]{DM_book_2nd_edition} and \cite[Theorem~2.6.22]{Geck_Malle_2020book}. Hill's analogue of Jordan decomposition is very useful for the quantitative information needed in computing representation zeta functions; see \cite[Subsection~5.2]{OnnPrasadSingla_2025_zetaA2poschar}. Recently it also plays an important role in proving the equivalence, in the case of $\mathrm{GL}_n$, between regular semisimple orbit representations and regular higher level Deligne--Lusztig representations; see \cite[Proposition~3.9 and Remark~4.6]{ChenStasinski_2023_algebraisation_II}.

\vspace{2mm} In this work, motivated by the constructions in \cite{Gerardin1973GL_n,ChenStasinski_2016_algebraisation,Chen_2016_GenericCharSh,ChenStasinski_2023_algebraisation_II}, we find that Hill's existence assertion can be refined to a method of constructing irreducible representations of $\mathrm{GL}_n(\mathcal{O}_r)$, at least when $r$ is even: Attached to each pair $(s,\nu)$ where $s$ is a semisimple element in $\mathfrak{gl}_n(\mathbb{F}_q)=M_n(\mathbb{F}_q)$ and $\nu$ is a nilpotent orbit representation of the centraliser of $s$, we construct explicitly a representation $J_r(s,\nu)$ of $\mathrm{GL}_n(\mathcal{O}_r)$, which we call the Jordan induction. It is of the form
$$J_r(s,\nu):=
\mathrm{Ind}_{(L_sG_r^{l})^F}^{\mathrm{GL}_n(\mathcal{O}_r)}
\widetilde{\phi_{s}\otimes\nu},$$
in which $(L_sG_r^{l})^F$ is a group extension of the centraliser (by the so-called arithmetic radical), and $\widetilde{\phi_s\otimes \nu}$ is certain representation of it; see Definition~\ref{defi:Jordan ind}. When $s$ is regular and semisimple, $\nu$ will be imprimitive (Definition~\ref{defi:primitive}) and this construction specialises to G\'erardin's representations \cite{Gerardin1973GL_n}.

\vspace{2mm} The main properties of $J_r(s,\nu)$ are:
\begin{thmMain}[Theorem~\ref{thm:main}] Suppose that $r$ is even. Then one has:
\begin{enumerate}
	\item[(A)] (Irreducibility) The representations $J_r(s,\nu)$ are always irreducible.
	\item[(B)] (Uniqueness) The map $J_r(s,-)$ is an injection for every semisimple $s \in M_n(\mathbb{F}_q)$; moreover, 
	$$\mathrm{Im}(J_r(s,-))\cap \mathrm{Im}(J_r(s',-))=\emptyset$$
	if  $s,s'\in M_n(\mathbb{F}_q)$ are not $\mathrm{GL}_n(\mathbb{F}_q)$-conjugated.
	\item[(C)] (Exhaustion) One has 
	$$\mathrm{Irr}(\mathrm{GL}_n(\mathcal{O}_r))=\mathrm{Im}(J_r(-,-)).$$
	\item[(D)] (Dimension) For a given $s$, the ratio $\frac{\dim J_r(s,\nu)}{\dim\nu}$ is a constant independent of $\nu$.
\end{enumerate}
\end{thmMain}
As we will see, once the formalism of $J_r(s,\nu)$ is set up, this theorem essentially follows from Clifford theory. 

\vspace{2mm} The above result reduces the problem of constructing irreducible representations to the problem of constructing nilpotent orbit representations. We should mention that, however, constructing nilpotent orbit representations is in general notoriously difficult, and only partial results are available: Hill \cite{Hill_1993_Jordan,Hill_1994_nilpotent} proves that there is an explicit monomial representation whose irreducible constituents are precisely the nilpotent orbit representations; this has been generalised to general connected reductive groups in \cite{Chen_2019_flag_orbit}. For $\mathrm{GL}_3$ with $p>3$, all nilpotent orbit representations are constructed only recently by Onn--Prasad--Singla \cite{OnnPrasadSingla_2025_zetaA2poschar}; prior to it, the counting of representations have been given by Avni--Klopsch--Onn--Voll \cite{AvniKlopschOnnVoll_2016_similarity} under a further assumption on $p$. For $\mathrm{GL}_n$ with $n\geq4$, constructing all the nilpotent orbit representations is believed to be intractable (see \cite[Section~4]{Nagornyi_1978_GL3}). Nevertheless, all the \emph{regular} orbit representations (including in particular the regular nilpotent ones) have been obtained by Stasinski--Stevens \cite{Stasinski_Stevens_2016_regularRep} (in general) and by Krakovski--Onn--Singla \cite{Krakovski_Onn_Singla_regularchar_2018} (for odd $p$), by different methods. Moreover, for general connected reductive groups the recent works \cite{ChenStasinski_2016_algebraisation,ChenStasinski_2023_algebraisation_II} show that all the regular semisimple orbit representations can be obtained by taking the higher Deligne--Lusztig inductions, though in general the latter miss some nilpotent ones; see \cite{Sta2011ExtendedDL,Chen_2019_flag_orbit,Chen_2025_stability_higher_Coxeter_unipotent}.

\vspace{2mm} In Section~\ref{sec:main construction}, after accomplishing the formalism of Jordan induction and proving Theorem~\ref{thm:main}, we turn to a discussion on the graded additive group
$$A(\mathcal{O}_r)=\bigoplus_{n\geq0}\mathbb{Z}\mathrm{Irr}(\mathrm{GL}_n(\mathcal{O}_r))$$ 
following Zelevinsky \cite{Zelevinsky_1981_bk}. This group admits an algebra structure and a co-algebra structure, given by the Jordan induction. These two structures are linked by an adjunction, and as an algebra $A(\mathcal{O}_r)$ is generated by the \emph{cuboidal} representations (Definition~\ref{defi:cuboidal rep}), which play a role similar to that of the cuspidal representations in Harish-Chandra theory; see Proposition~\ref{prop:cuboidal rep}, \ref{prop:adjunction and primitivity}, and \ref{prop:Hopf axiom}. We remark that in this direction several related, but different, constructions have been studied, like the one in Crisp--Meir--Onn \cite{Crisp_Meir_Onn_varHarishChandra,Crisp_Meir_Onn_inductive_approach_GLn} using idempotent techniques, and the one in \cite{Chen_2024_Hopfalg_duality} using invariant characters of the Lie algebra $\mathfrak{gl}_n(\mathbb{F}_q)$; we expect to investigate their relations in a future work.

\vspace{2mm} In Section~\ref{sec:rep from orbits}  we consider a simple application of Jordan induction to the sections of the orbit map $\Omega(-)$ (Definition~\ref{defi:orbit type}), in the equal characteristic case. More precisely, we will construct a map
$$
\mathcal{J}_r \colon \mathfrak{gl}_n(\mathbb{F}_q)\longrightarrow \mathrm{Irr}(\mathrm{GL}_n(\mathcal{O}_r))
$$
satisfying that 
$$\omega \in \Omega(\mathcal{J}_r(\omega))$$ 
for every $\omega\in \mathfrak{gl}_n(\mathbb{F}_q)$. This relies on a property of nilpotent matrices and the construction of Jordan induction; see Corollary~\ref{coro:from orbit to irrep}.

\vspace{2mm} In Section~\ref{sec:cohom} we turn to the cohomological nature of $J_r(-,-)$. In the classical finite field case, an important property of Lusztig induction is that it gives a realisation of Lusztig's Jordan decomposition; see \cite[Theorem~11.4.3(ii)]{DM_book_2nd_edition}. In Conjecture~\ref{conj:Jordan=Lusztig} we propose an analogue for $J_r(-,-)$, which asserts that
$$J_r(s,\nu)
\cong R_{L_s,U_s}^{\mathrm{GL}_n}(\phi_s\otimes\nu)$$
for every parameter $(s,\nu)$, where $R_{L_s,U_s}^{\mathrm{GL}_n}$ denotes the Lusztig induction (the parahoric version introduced by Chan \cite{Chan_2024_Scalar_Product}). This generalises the algebraisation theorem (specialising to $\mathrm{GL}_n$) of \cite{ChenStasinski_2016_algebraisation,ChenStasinski_2023_algebraisation_II} from regular semisimple orbit representations to all irreducible representations. We verify a few cases in Proposition~\ref{prop:simplest cases of J=L conj}. By Proposition~\ref{prop:role of cuboidal in cohom rep} this conjecture reduces to the cuboidal case.

\vspace{2mm} In Section~\ref{sec:odd levels} we discuss a possible construction of $J_r(-,-)$ for odd $r>1$. This involves a conjectural non-abelian Heisenberg lift inspired by the works \cite{ChenStasinski_2023_algebraisation_II,Stasinski_Stevens_2016_regularRep,Bushnell_Froelich_book_GaussSum_1983}; see Conjecture~\ref{conj:odd levels}. 

\vspace{2mm} Presented in the final Section~\ref{sec:GL3 r=2} is an example on the construction of irreducible representations of $\mathrm{GL}_3(\mathcal{O}_2)$ within the framework of Jordan induction.

\vspace{2mm} \noindent {\bf Acknowledgement.} The author is grateful to Alexander Stasinski for helpful comments on an earlier version of this paper, and also for suggesting him to read Hill's series of works during his PhD studies; it was only after graduating that he began to truly appreciate these works. Part of this work was undertaken during the author's stay at TSIMF in Jan 2026, for the conference Geometric Representation Theory and Automorphic Forms, and the author thanks the organisers for the hospitality.

\vspace{2mm} \noindent {\bf Human disclosure.} Except for a few grammar corrections and \LaTeX code search assisted by LLM tools, this work is done by the carbon-based human male.

\section{Jordan induction}\label{sec:main construction}

We start by briefly setting up the notation, partially following our previous work \cite{ChenStasinski_2016_algebraisation}.

\vspace{2mm} First, fix a maximal unramified extension $K^{\mathrm{ur}}$ of $K:=\mathrm{Frac}(\mathcal{O})$, and let $\mathcal{O}^{\mathrm{ur}}$ be its ring of integers; we also write $\mathcal{O}^{\mathrm{ur}}_r:=\mathcal{O}^{\mathrm{ur}}/\pi^r$. For $F$ the Frobenius element ``$x\mapsto x^q$'' in $\mathrm{Gal}(\overline{\mathbb{F}}_q/\mathbb{F}_q)$, by identifying $\mathrm{Gal}(\overline{\mathbb{F}}_q/\mathbb{F}_q)$ with  $\mathrm{Gal}(K^{\mathrm{ur}}/K)$ we also write $F$ for the corresponding element in $\mathrm{Gal}(K^{\mathrm{ur}}/K)$. The Greenberg functor (see \cite{Greenberg19611,Greenberg19632,Sta2009Unramified}) identifies $\mathrm{GL}_n(\mathcal{O}^{\mathrm{ur}}_r)$ with (the group of $\overline{\mathbb{F}}_q$-points of) an algebraic group $G_r$ over $\overline{\mathbb{F}}_q$ such that 
$$G_r(\overline{\mathbb{F}}_q)\cong \mathrm{GL}_n(\mathcal{O}^{\mathrm{ur}}_r)$$ 
as abstract groups. This algebraic group is defined over $\mathbb{F}_q$ and the geometric Frobenius, which commutes with the above isomorphism, should also be denoted by $F$. In particular, $G_r^F\cong\mathrm{GL}_n(\mathcal{O}_r)$. For every $i,j\in\{ 1,...,r \}$ with $i\leq j$, let $\rho_{j,i}$ be the reduction map from $\mathcal{O}^{\mathrm{ur}}_j$ to $\mathcal{O}^{\mathrm{ur}}_i$. They induce surjective ring morphisms
$$M_n(\mathcal{O}^{\mathrm{ur}}_j)\longrightarrow M_n(\mathcal{O}^{\mathrm{ur}}_i)$$ 
which we denote again by $\rho_{j,i}$. When $j=r$ we write $\rho_i$ short for $\rho_{r,i}$; the reduction map from $\mathcal{O}^{\mathrm{ur}}$ will also be written as $\rho_i$. If $H\subseteq G_j$ is a closed subgroup, we use the notation $H_i:=\rho_{j,i}(H)$ and $H^i:=H\cap\mathrm{Ker}(\rho_{j,i})$ for convention, when there is no confusion. Let $l:=[\frac{r+1}{2}]$ and $l':=[\frac{r}{2}]$.  From now on we fix  a  group morphism 
 \begin{equation}\label{temp:psi}
       \psi\colon \mathcal{O}_{l'}\rightarrow \overline{\mathbb{Q}}_{\ell}^{\times}
 \end{equation}
that is non-trivial on $\pi^{l'-1}\mathcal{O}_{l'}$, where   $\ell$ is a prime invertible in $\mathbb{F}_q$.

\vspace{2mm} For a given semisimple element $s\in \mathfrak{gl}_n(\mathbb{F}_q)=M_n(\mathbb{F}_q)$, by conjugating $s$ into the diagonal part of $M_n(\overline{\mathbb{F}}_q)$, up to permutation we have (here  ``$\rightsquigarrow$'' means ``conjugates to'')
\begin{equation*}
\begin{split}
s
\rightsquigarrow\mathrm{diag}& ((s_1I_{n_1},s_1^qI_{n_1},\cdots, s_1^{q^{\lambda_1-1}}I_{n_1}),
\cdots,
(s_dI_{n_d},s_d^qI_{n_d},\cdots, s_d^{q^{\lambda_d-1}}I_{n_d}) ),\\
\rightsquigarrow\mathrm{diag}&(
\underbrace{(s_1,s_1^q,\cdots, s_1^{q^{\lambda_1-1}}),
\cdots,(s_1,s_1^q,\cdots, s_1^{q^{\lambda_1-1}})}_{n_1\ \textrm{times}},\cdots\cdots,\\
&\cdots\cdots, \underbrace{(s_d,s_d^q,\cdots, s_d^{q^{\lambda_d-1}}),
\cdots,(s_d,s_d^q,\cdots, s_d^{q^{\lambda_d-1}})}_{n_d\ \textrm{times}}
),
\end{split}
\end{equation*}
where the $s_i$'s are in $\overline{\mathbb{F}}_q$ and
the polynomials $\prod_{j=0}^{\lambda_{i}-1}(x-s_{i}^{q^j})$ (for $i=1,\cdots, d$) are pairwise distinct irreducible polynomials over $\mathbb{F}_q$. Note in particular that 
\begin{equation}\label{temp:eigenpoly of s}
E_s(x):=\prod_{i=1}^d \left(\prod_{j=0}^{\lambda_{i}-1}(x-s_{i}^{q^j})\right)^{n_i}
\end{equation}
is the characteristic polynomial of $s$.

\vspace{2mm} For each $i\in\{1,\cdots,d\}$ let $K^{(\lambda_i)}$ be the unramified extension of $K$ of degree $\lambda_i$ inside $K^{\mathrm{ur}}$,  whose residue field $\mathbb{F}_{q^{\lambda_i}}$ is generated by $s_i$.  For each $i$ we pick (and fix) a representative $\hat{s}_i$ of $s_i$ in $\mathcal{O}^{(\lambda_i)}$ (the ring of integers in $K^{(\lambda_i)}$). So
\begin{equation*}
\begin{split}
\hat{s}
=\mathrm{diag}&(
\underbrace{(\hat{s}_1,F(\hat{s}_1),\cdots, F^{\lambda_1-1}(\hat{s}_1)),
\cdots,(\hat{s}_1,F(\hat{s}_1),\cdots, F^{\lambda_1-1}(\hat{s}_1))}_{n_1\ \textrm{times}},\\
&\cdots\cdots, 
\underbrace{(\hat{s}_d,F(\hat{s}_d),\cdots, F^{\lambda_d-1}(\hat{s}_d)),
\cdots,(\hat{s}_d,F(\hat{s}_d),\cdots, F^{\lambda_d-1}(\hat{s}_d))}_{n_d\ \textrm{times}}
),
\end{split}
\end{equation*}
is a lift of a conjugate of $s$ in $M_n(\mathcal{O}^{\mathrm{ur}})$. 

\begin{lemm}\label{lemm: conjugation to rational matrix}
Each diagonal block matrix $\mathrm{diag}(\hat{s}_i,F(\hat{s}_i),\cdots, F^{\lambda_i-1}(\hat{s}_i))$ in above can be conjugated by a matrix in ${\mathrm{GL}}_{\lambda_i}(\mathcal{O}^{\mathrm{ur}})$ to a (regular and semisimple) matrix $A_i\in M_{\lambda_i}(\mathcal{O})$; the ring of   matrices commuting with $A_i$ in $M_{\lambda_i}(\mathcal{O})$ is isomorphic to $\mathcal{O}^{(\lambda_i)}$.
\end{lemm}
\begin{proof}
Let $c\in \mathrm{GL}_{\lambda_i}(\mathcal{O})$ be the permutation  matrix such that 
$$c\cdot F(\mathrm{diag}(\hat{s}_i,F(\hat{s}_i),\cdots, F^{\lambda_i-1}(\hat{s}_i)))\cdot c^{-1}
=
\mathrm{diag}(\hat{s}_i,F(\hat{s}_i),\cdots, F^{\lambda_i-1}(\hat{s}_i)).$$
For every $m\in\mathbb{Z}_{>0}$, by the Greenberg functor we can view $\mathrm{GL}_{\lambda_i}(\mathcal{O}^{\mathrm{ur}}_m)$ as a connected algebraic group $\Gamma_m$ over $\overline{\mathbb{F}}_q$. Let $L_m\colon x\mapsto x^{-1}F(x)$ be the Lang isogeny on $\Gamma_m$. 

\vspace{2mm} Now fix an $m$. Then by the Lang--Steinberg theorem we can find $g_m\in \Gamma_m$ with $L_m(g_m)=c$. As $c$ is a permutation matrix, we can naturally lift it to $\Gamma_{m+1}$. Then again by the Lang--Steinberg theorem we can pick $g'_{m+1}\in\Gamma_{m+1}$ with $L_{m+1}(g'_{m+1})=c$. Since 
$$L_{m}(\rho_{m+1,m}(g'_{m+1}))=\rho_{m+1,m}(L_{m+1}(g'_{m+1}))=c=L_m(g_m),$$
we have $g_m=h_m \cdot \rho_{m+1,m}(g'_{m+1})$ for some $h_m\in \Gamma_m^F$. Let $h_{m+1}\in \Gamma_{m+1}^F$ be such that $\rho_{m+1,m}(h_{m+1})=h_m$ and put $g_{m+1}:=h_{m+1}\cdot g'_{m+1}$. Note that we have $\rho_{m+1,m}(g_{m+1})=g_m$ and $L_{m+1}(g_{m+1})=c$.

\vspace{2mm} Repeating the above process inductively from $m=1$ we get a sequence $(g_m)_m$ satisfying that $L_m(g_m)=c$ and $\rho_{m,m'}(g_m)=g_{m'}$ for all $m\geq m'\geq1$.

\vspace{2mm} Now, since 
$$\mathrm{GL}_{\lambda_i}(\mathcal{O}^{\mathrm{ur}})=\varprojlim_{m} \Gamma_m,$$
the sequence $(g_m)_m$ gives an element $g\in \mathrm{GL}_{\lambda_i}(\mathcal{O}^{\mathrm{ur}})$ with  $g^{-1}F(g)=c$. So the matrix
$$A_i:=g\cdot \mathrm{diag}(\hat{s}_i,F(\hat{s}_i),\cdots, F^{\lambda_i-1}(\hat{s}_i))\cdot g^{-1}$$
is $F$-stable, and hence lies in $M_{{\lambda_i}}(\mathcal{O})$. Now, if $B\in M_{{\lambda_i}}(\mathcal{O})$ satisfies that $A_iB=BA_i$, then, since the elements $F^j(\hat{s}_i)$ are pairwise distinct, the matrix $g^{-1}Bg\in M_{\lambda_i}(\mathcal{O}^{\mathrm{ur}})$ must be a diagonal matrix. Moreover, as $B$ is $F$-stable we have
$$F(g^{-1}Bg)=F(g^{-1})gg^{-1}Bgg^{-1}F(g)=c^{-1}g^{-1}Bgc.$$
This implies that $g^{-1}Bg$ is of the form
$$\mathrm{diag}(t_i,F(t_i),\cdots, F^{\lambda_i-1}(t_i))$$
for some $F^{\lambda_i}(t_i)=t_i$ (i.e., $t_i\in\mathcal{O}^{(\lambda_i)}$), which gives the second assertion.
\end{proof}

By Lemma~\ref{lemm: conjugation to rational matrix}, $\hat{s}$ can be conjugated (by a matrix in $\mathrm{GL}_n(\mathcal{O}^{\mathrm{ur}})$) to
$$
A_s:=\mathrm{diag}( 
(\underbrace{A_1,\cdots, A_1}_{n_1\ \text{times}}),
\cdots,
(\underbrace{A_d,\cdots, A_d}_{n_d\ \textrm{times}})
)\in M_n(\mathcal{O}),
$$
and the second assertion of Lemma~\ref{lemm: conjugation to rational matrix} gives an isomorphism of rings
\begin{equation}\label{temp:s.s. centraliser}
 C_{M_n(\mathcal{O})}(A_s)\cong
\prod_{i=1}^dM_{n_i}(\mathcal{O}^{(\lambda_i)});
\end{equation}
similarly we have
\begin{equation}\label{temp:s.s. centraliser level m}
 C_{M_n(\mathcal{O}_{m})}(\rho_{m}(A_s))\cong
\prod_{i=1}^dM_{n_i}(\mathcal{O}^{(\lambda_i)}/\pi^m)
\end{equation}
for all $m\in\mathbb{Z}_{>0}$. In the case $m=r$ let
\begin{equation}\label{temp:Lie alg of L over ring}
\mathcal{L}_s:=C_{M_n(\mathcal{O}_{r})}(\rho_{r}(A_s)).
\end{equation}
By \eqref{temp:s.s. centraliser level m}, if we put

\begin{equation}\label{temp: defn of whole centraliser}
L_s:=C_{\mathrm{GL}_n(\mathcal{O}^{\mathrm{ur}}_r)}(\rho_{r}(A_s))\ 
\left(
    \cong
    \prod_{i=1}^d(\mathrm{GL}_{n_i}(\mathcal{O}^{\mathrm{ur}}_r))^{\lambda_i}
\right),
\end{equation}
then
\begin{equation}\label{temp: defn of rational centraliser}
L_s^F=C_{\mathrm{GL}_n(\mathcal{O}_r)}(\rho_{r}(A_s))
\cong
\prod_{i=1}^d\mathrm{GL}_{n_i}(\mathcal{O}^{(\lambda_i)}/\pi^r).
\end{equation}
From the argument of Lemma~\ref{lemm: conjugation to rational matrix}, the isomorphisms \eqref{temp:s.s. centraliser},\eqref{temp:s.s. centraliser level m},\eqref{temp: defn of whole centraliser},\eqref{temp: defn of rational centraliser} are obtained via taking conjugations, and hence commute with taking reduction maps.

\vspace{2mm} Now the point is, the semisimple element $s$ not only brings us a subgroup $L_s^F\subseteq\mathrm{GL}_n(\mathcal{O}_r)$, but also brings a specific $1$-dimensional representation of $L_s^F$: First, let $a$ denote the identification of additive groups 
$$a\colon \pi^{l}\mathcal{O}^{\mathrm{ur}}_r \cong  \mathcal{O}^{\mathrm{ur}}_{l'}; 
\quad \pi^l x \longmapsto \rho_{l'}(x).$$
When there is no confusion, we also denote by $a$ the identification of matrix groups 
\begin{equation}\label{eqn: defn of a(-)}
a\colon G_r^l
\cong  
M_n(\mathcal{O}^{\mathrm{ur}}_{l'}); 
\quad I_n+\pi^l g \longmapsto \rho_{l'}(g).
\end{equation}
Then we get a representation
$$\phi'_{s_i}(x):=\psi\left(\sum_{j=0}^{\lambda_i-1} F^j(\rho_{l'}(\hat{s}_i)a(x))\right)$$
of  $\pi^{l}\mathcal{O}^{(\lambda_i)}/\pi^r$; since   $\pi^{l}\mathcal{O}^{(\lambda_i)}/\pi^r$ can be viewed as a subgroup of $(\mathcal{O}^{(\lambda_i)}/\pi^r)^{\times}$ (by formally taking $\pi^{l}x$ to  $1+\pi^{l}x$), we can consider the induced representation $\mathrm{Ind}_{\pi^{l}\mathcal{O}^{(\lambda_i)}/\pi^r}^{(\mathcal{O}^{(\lambda_i)}/\pi^r)^{\times}}\phi'_{s_i}$. For each $s_i$ we  fix an (arbitrary) irreducible constituent $\phi''_{s_i}$ of $\mathrm{Ind}_{\pi^{l}\mathcal{O}^{(\lambda_i)}/\pi^r}^{(\mathcal{O}^{(\lambda_i)}/\pi^r)^{\times}}\phi'_{s_i}$. This in turn gives a $1$-dimensional representation 
$$\phi_{s}(-):=\prod_{i=1}^d\phi''_{s_i}\left( \det(-)_i \right)$$
of $L_s^F$, in which the symbol $\det(-)_i$ means taking the determinant of  $\mathrm{GL}_{n_i}(\mathcal{O}^{(\lambda_i)}/\pi^r)$. 

\vspace{2mm} Since $\mathrm{Tr}(-,-)$ is a non-degenerate bilinear form on $M_n(\mathcal{O}_{l'})$, the irreducible representations of $(G_r^{l})^F$ are of the form (see e.g.\ \cite[Remark~3.6]{ChenStasinski_2016_algebraisation})
\begin{equation}\label{temp:char of Lie alg}
    \psi_{\beta}(-):=\psi(\mathrm{Tr}(\beta\cdot a(-))),
\end{equation}
where $\beta\in M_n(\mathcal{O}_{l'})$ and $a(-)$ is defined by \eqref{eqn: defn of a(-)}. Similarly, for $B\in M_n(\mathbb{F}_q)$ we put
$$\psi_{B}(-)=\psi_{\hat{B}}|_{(G_r^{r-1})^F}(-),$$
where $\hat{B}$ is a lift of $B$ to $M_n(\mathcal{O}_{l'})$ 
(by construction $\psi_{B}$ is independent of the choice of such a lift). Now Clifford theory tells that, for every $\sigma\in\mathrm{Irr}(G_r^F)$, the restriction $\sigma|_{(G_r^{r-1})^F}$ is the sum of characters $\psi_{B}$ with $B$ running over a $\mathrm{GL}_n(\mathbb{F}_q)$-conjugacy class $\Omega(\sigma)\subseteq M_n(\mathbb{F}_q)$.
\begin{defi}\label{defi:orbit type}
    Given $\sigma\in\mathrm{Irr}(G_r^F)$, we call $\Omega(\sigma)$ the orbit of $\sigma$. Moreover, $\sigma$ is said to be a regular (resp.\ nilpotent, semisimple) orbit representation, if $\Omega(\sigma)$ contains a regular (resp.\ nilpotent, semisimple) element.
\end{defi}
We can and shall extend the construction of representation orbits to $\mathrm{Irr}(L_s^F)$, and when there is no confusion we use the same notation $\Omega(-)$.

\begin{lemm}\label{lemm:s.s. part of orbit_prep}
We have $\Omega(\phi_s)=\{\rho_1(A_s)\}$.
\end{lemm}
\begin{proof}
Let $g\in (L_s^{r-1})^F$; by conjugating it to be a $d\times d$ blocked matrix according to \eqref{temp: defn of rational centraliser}, we can write it as $\mathrm{diag}(g_1,\cdots,g_d)$ where $g_i\in I_{n_i}+\pi^{r-1}M_{n_i}(\mathcal{O}^{(\lambda_i)}/\pi^{r})$. Let $t_i\in M_{n_i}(\mathcal{O}^{(\lambda_i)}/\pi^{r})$ be such that $g_i-I_{n_i}=\pi^{r-1}t_i$ ($i=1,\cdots, d$). We have (by definition of determinant and by   $\pi^{r-1}\cdot\pi^{r-1}=0$)
$$\det(g)_i=\det(g_i)=1+\pi^{r-1} \mathrm{Tr}(t_i).$$
So
\begin{equation}\label{temp: s.s. orbit of 1-dim char}
\begin{split}
\phi_s(g)
&=\prod_{i=1}^d\psi\left(\sum_{j=0}^{\lambda_i-1} F^j\left(\rho_{l'}(\hat{s}_i)\cdot a(\pi^{r-1}\mathrm{Tr}(t_i))\right)\right)\\
&=\psi\left(\sum_{i=1}^d\sum_{j=0}^{\lambda_i-1} F^j\left(\rho_{l'}(\hat{s}_i)\cdot a(\pi^{r-1}\mathrm{Tr}(t_i))\right)\right)\\
&=\psi\left(\sum_{i=1}^d\sum_{j=0}^{\lambda_i-1} 
F^j(
    \mathrm{Tr} 
    \left(  
        \rho_{l'}(\hat{s}_i)I_{n_i} \cdot a(g_i)
    \right)
    )
\right).
\end{split}
\end{equation}
For $m\in\mathbb{Z}_{>0}$, let $a'$ denote the identification $I_m+M_m(\pi^{r-1}\mathcal{O}^{\mathrm{ur}}_{r})\cong M_m(\overline{\mathbb{F}}_q)$ (from a multiplicative group to an additive group). Then
\begin{equation}\label{temp: s.s. orbit of 1-dim char 2}
\eqref{temp: s.s. orbit of 1-dim char}
=\psi
\left(
\pi^{l'-1}\sum_{i=1}^d
\mathrm{Tr}_{\mathbb{F}_{q^{\lambda_i}}/\mathbb{F}_q}
(
    \mathrm{Tr} 
    \left(  
        s_iI_{n_i} \cdot a'(g_i)
    \right)
)
\right).
\end{equation}
So, conjugating $\mathrm{diag}(g_1,...,g_d)$ back to $g\in (L_s^{r-1})^F\subseteq (G_r^{r-1})^F$,
\eqref{temp: s.s. orbit of 1-dim char 2} 
becomes $\psi_{t}(g)$ for some $t\in M_n(\mathbb{F}_q)$ that can be conjugated to 
$\mathrm{diag}(s_1I_{n_1},...,s_1^{q^{\lambda_1-1}}I_{n_1},...,s_d^{q^{\lambda_d-1}}I_{n_d})$; the latter matrix is by definition conjugated to $\rho_1(A_s)$, a central element in $L_s^F$, and so the assertion follows.
\end{proof}

More generally:

\begin{lemm}\label{lemm:s.s. part of orbit}
Let $\nu$ be a nilpotent orbit representation  of $L_s^F$. Then $\Omega({\phi_{s}\otimes \nu})=\rho_1(A_s)+\Omega(\nu)$; in particular, the semisimple parts of the elements in $\Omega({\phi_{s}\otimes \nu})$ are always $\rho_1(A_s)$.
\end{lemm}
\begin{proof}
This follows from Lemma~\ref{lemm:s.s. part of orbit_prep}.
\end{proof}

For  $\nu$ a nilpotent orbit representation  of $L_s^F$, let $\widetilde{\phi_s\otimes\nu}$ be the inflation along the projection $(L_sG_r^l)^F\rightarrow L_s^F$ with respect to the Iwahori decomposition (see e.g.\ \cite[Lemma~2.2]{Sta2009Unramified}).

\begin{lemm}\label{lemm: stab_2 level}
The stabiliser of $\widetilde{\phi_s\otimes \nu}|_{(G_r^l)^F}$ in $G_r^F$ is $(L_sG_r^{l'})^F$.
\end{lemm}
\begin{proof}
Since $\widetilde{\phi_s\otimes \nu}$ is defined on $(L_sG_r^{l})^F$, and since $(G_r^{l'})^F$ commutes with $(G_r^{l})^F$, we have 
$$(L_sG_r^{l'})^F=(L_sG_r^{l})^F  (G_r^{l'})^F 
\subseteq
\mathrm{Stab}_{G_r^F}\left(\widetilde{\phi_s\otimes \nu}|_{(G_r^l)^F}\right).$$
So it suffices to show that 
\begin{equation}\label{temp in lemma: inclusion of stab}
\rho_{l'}\left(\mathrm{Stab}_{G_r^F}\left(\widetilde{\phi_s\otimes \nu}|_{(G_r^l)^F}\right)\right)\subseteq \rho_{l'}(L_s).
\end{equation}
Recall that every irreducible representation of $(G_r^l)^F$ is of the form (see \eqref{temp:char of Lie alg})
$$\psi_{\beta}(x):=\psi\left(\mathrm{Tr} (\beta\cdot a(x)) \right)$$
for some $\beta\in M_n(\mathcal{O}_{l'})$. We shall try to understand $\widetilde{\phi_s\otimes \nu}|_{(G_r^l)^F}$ through this viewpoint.

\vspace{2mm} Since $\phi_s\otimes \nu$ is an irreducible representation of $L_s^F$, by Clifford theory and  similar arguments of Lemma~\ref{lemm:s.s. part of orbit_prep} and  Lemma~\ref{lemm:s.s. part of orbit} we can write
\begin{equation}\label{temp:orbit sum form on level l ker}
    \widetilde{\phi_s\otimes \nu}|_{(G_r^l)^F}
    =\lambda\cdot\sum_{x\in L_s^F} \psi_{{^x\beta(s,\nu)}};
\end{equation}
here $\lambda\in\mathbb{Z}_{>0}$ and
\begin{equation}\label{temp:defn of beta}
\beta(s,\nu):=\rho_{l'}(A_s)+y_{\nu}
\end{equation}
where $y_{\nu}\in \rho_{l'}(\mathcal{L}_s^F)$ is a lift of some element in  $\Omega(\nu)\subseteq \rho_{1}(\mathcal{L}_s^F)$ (here $\mathcal{L}_s$ is defined in \eqref{temp:Lie alg of L over ring}). Thus to prove  \eqref{temp in lemma: inclusion of stab} it is equivalent to prove: If  $g \in G_{l'}^F$ normalises the set $\{ {^x\beta(s,\nu)}\mid x\in \rho_{l'}(L_s)^F \}$, namely
$$g \cdot  \{ {^x\beta(s,\nu)}\mid x\in \rho_{l'}(L_s^F) \} \cdot  g^{-1}\subseteq  \{ {^x\beta(s,\nu)}\mid x\in \rho_{l'}(L_s^F)\},$$
then
\begin{equation}\label{temp: assertion on g}
g\in \rho_{l'}(L_s).
\end{equation}
Note that \eqref{temp: assertion on g} is equivalent to
\begin{equation}\label{temp: assertion on g (2)}
\rho_{l',m}(g)\in \rho_{m}(L_s)\quad\forall m\in \{ 1,2,\cdots,l' \}.
\end{equation}
In the below we   prove \eqref{temp: assertion on g (2)} by an induction on $m$.

\vspace{2mm} First, by construction we have
\begin{equation*}
{^g(\rho_{l'}(A_s))}+{^gy_{\nu}}
={^x(\rho_{l'}(A_s))}+{^xy_{\nu}}
=\rho_{l'}(A_s)+{^xy_{\nu}}
\end{equation*}
for some $x\in \rho_{l'}(L_s^F)$.
So  the uniqueness of Jordan decomposition (in $M_n(\mathbb{F}_q)$) implies that 
$${^{\rho_{l',1}(g)}(\rho_{1}(A_s))} 
=\rho_{1}(A_s).$$
Thus $\rho_{l',1}(g)\in \rho_{1}(L_s)$; that is, \eqref{temp: assertion on g (2)} holds for $m=1$.

\vspace{2mm} Now assume that \eqref{temp: assertion on g (2)} holds for some $m\in\{1,2,\cdots,l'-1 \}$. We shall show that \eqref{temp: assertion on g (2)} holds for $m+1$.

\vspace{2mm} Let $g'\in \rho_{m+1}(L_s)$ be a lift of $\rho_{l',m}(g)\in \rho_{m}(L_s)$; by construction we have 
$$g'^{-1}\rho_{l',m+1}(g)=I+\pi^{m}g''$$
for some $g''\in M_n(\mathcal{O}_{m+1})$. It remains to show that  
\begin{equation}\label{temp: reduction to level 1}
\rho_{m+1,1}(g'')\in \rho_{1}(\mathcal{L}_s)
\end{equation}
(as this would imply $\rho_{l',m+1}(g)=g'(I+\pi^{m}g'')\in \rho_{m+1}(L_s)$).

\vspace{2mm} As both $g'$ and $\rho_{l',m+1}(g)$ normalise $\rho_{l',m+1}(\{ {^x\beta(s,\nu)}\mid x\in \rho_{l'}(L_s^F) \})$, so does $I+\pi^mg''$. Therefore
\begin{equation*}
{^{I+\pi^mg''}\rho_{l',m+1}(\beta(s,\nu))}
={^x\rho_{l',m+1}(\beta(s,\nu))}
\end{equation*}
for some $x\in \rho_{m+1}(L_s^F)$; that is, we have
\begin{equation*}
\rho_{l',m+1}(\beta(s,\nu))+\pi^m(g''\cdot\rho_{l',m+1}(\beta(s,\nu))-\rho_{l',m+1}(\beta(s,\nu))\cdot g'')
={^x(\rho_{l',m+1}(\beta(s,\nu)))}
\end{equation*}
for some $x\in \rho_{m+1}(L_s^F)$. This implies that
$$\pi^m(g''\cdot \rho_{l',m+1}(\beta(s,\nu))-\rho_{l',m+1}(\beta(s,\nu))\cdot g'')\in \rho_{m+1}(\mathcal{L}_s),$$
or equivalently,
\begin{equation}\label{temp: adjoint action}
\rho_{m+1,1}(g'')\cdot \rho_{l',1}(\beta(s,\nu))-\rho_{l',1}(\beta(s,\nu))\cdot \rho_{m+1,1}(g'')\in \rho_{1}(\mathcal{L}_s).
\end{equation}
Note that the left side of \eqref{temp: adjoint action} can be written as $\mathrm{ad}_{\rho_{m+1,1}(g'')}(\rho_{l',1}(\beta(s,\nu)))$ (adjoint action in Lie algebra). To continue we  simplify the notation a bit: Let $A:=\rho_{1}(A_s)$, $N:=\rho_{l',1}(y_{\nu})$, $g''':=\rho_{m+1,1}(g'')$,  $\mathcal{G}:=M_n(\overline{\mathbb{F}}_q)$, and $\mathcal{L}:=\rho_{1}(\mathcal{L}_s)$. Then \eqref{temp: adjoint action} becomes
\begin{equation}\label{temp_formula:adjoint action_clean form}
\mathrm{ad}_{g'''}(A+N)\in \mathcal{L},
\end{equation}
and our target \eqref{temp: reduction to level 1} becomes to showing that
\begin{equation}\label{temp: reduction to level 1 new form}
g'''\in \mathcal{L}.
\end{equation}
Since $A$ is semisimple, the linear transformation $\mathrm{ad}_A$ is diagonalisable (over $\overline{\mathbb{F}}_q$). So we can consider the eigenspace decomposition 
$$\mathcal{G}=\mathcal{L}\oplus\left(\bigoplus_{\lambda\in\overline{\mathbb{F}}_q^{\times}}\mathcal{G}_{\lambda}\right)$$
with respect to  $\mathrm{ad}_A$.
Let $g'''=g_L+\sum_{\lambda\neq0}g_{\lambda}$ be the corresponding decomposition. Then \eqref{temp_formula:adjoint action_clean form} implies that
$$\sum_{\lambda\neq0}\left(\mathrm{ad}_{g_{\lambda}}(A)+\mathrm{ad}_{g_{\lambda}}(N)\right)\in \mathcal{L},$$
or equivalently
\begin{equation}\label{temp_formula: each root}
\sum_{\lambda\neq0} \left(-\lambda g_{\lambda} - \mathrm{ad}_{N}(g_{\lambda})\right)\in \mathcal{L}.
\end{equation}
Consider the Jacobi identity
$$[A,[N,g_{\lambda}]]+[g_{\lambda},[A,N]]+[N,[g_{\lambda},A]]=0;$$
since $N$ commutes with $A$, this implies that
$$\mathrm{ad}_A(\mathrm{ad}_N(g_{\lambda}))=\lambda\mathrm{ad}_N(g_{\lambda}),$$
which means that $\mathrm{ad}_N(g_{\lambda})\in\mathcal{G}_{\lambda}$. So by \eqref{temp_formula: each root}  we see 
\begin{equation}\label{temp_formula: each root (2)}
\lambda g_{\lambda}+\mathrm{ad}_{N}(g_{\lambda})=0
\end{equation}
for every non-zero eigenvalue $\lambda$. 

\vspace{2mm} If some $g_{\lambda}$ is non-zero, then \eqref{temp_formula: each root (2)} tells that $g_{\lambda}$ is an eigenvector of $\mathrm{ad}_{N}$ with the eigenvalue $-\lambda$, but this is impossible because all the eigenvalues  of $\mathrm{ad}_N$ are zero (see e.g.\ \cite[3.2~Lemma]{Humphreys_intro_Lie_RepThy}). Therefore all those $g_{\lambda}$ are zero. Hence
$$g'''=g_L\in\mathcal{L},$$
as desired in \eqref{temp: reduction to level 1 new form}.
\end{proof}

\begin{defi}\label{defi:Jordan ind}
Let the notation be as above, and suppose that $r$ is even. For a semisimple element $s\in M_n(\mathbb{F}_q)$ and a nilpotent orbit representation $\nu$ of $L_s^F$, we put
$$J_r(s,\nu)=
\mathrm{Ind}_{(L_sG_r^{l})^F}^{\mathrm{GL}_n(\mathcal{O}_r)}
\widetilde{\phi_{s}\otimes\nu}.$$
We call $J_r(-,-)$ the level $r$ Jordan induction.
\end{defi}

\begin{thm}\label{thm:main}
Let the notation be  as above, and suppose that $r$ is even. Then:
\begin{enumerate}
\item[(A)] (Irreducibility) The representations $J_r(s,\nu)$ are always irreducible.
\item[(B)] (Uniqueness) The map $J_r(s,-)$ is an injection for every semisimple $s \in M_n(\mathbb{F}_q)$; moreover, 
$$\mathrm{Im}(J_r(s,-))\cap \mathrm{Im}(J_r(s',-))=\emptyset$$
if  $s,s'\in M_n(\mathbb{F}_q)$ are not $\mathrm{GL}_n(\mathbb{F}_q)$-conjugated.
\item[(C)] (Exhaustion) One has 
$$\mathrm{Irr}(\mathrm{GL}_n(\mathcal{O}_r))=\mathrm{Im}(J_r(-,-)).$$
\item[(D)] (Dimension) For a given $s$, the ratio $\frac{\dim J_r(s,\nu)}{\dim\nu}$ is a constant independent of $\nu$.
\end{enumerate}
\end{thm}

\begin{proof}
(A) This follows  from Clifford theory (see e.g.\ \cite[(6.11)~Theorem~(a)]{Isaacs_CharThy_Book}) and Lemma~\ref{lemm: stab_2 level}.

\vspace{2mm}
\noindent (B) If $s$ and $s'$ are not $\mathrm{GL}_n(\mathbb{F}_q)$-conjugated, then the semisimple parts of elements in the corresponding orbits of $J_r(s,-)$ and of $J_r(s',-)$ are not $\mathrm{GL}_n(\mathbb{F}_q)$-conjugated (by Lemma~\ref{lemm:s.s. part of orbit}), and hence the first assertion follows.

\vspace{2mm} Now we turn to the second assertion on the injectivity. If $J_r(s,\nu_1)\cong J_r(s,\nu_2)$, then by Clifford theory \cite[(6.2)~Theorem]{Isaacs_CharThy_Book} we have ${^g\beta(s,\nu_1)}={\beta(s,\nu_2)}$ for some $g\in G_{l}^F$, where $\beta(s,\nu_i)$ is defined as in \eqref{temp:defn of beta}. Note that the same inductive argument in the proof of Lemma~\ref{lemm: stab_2 level} gives that $g\in L_s^F$, which implies that (via \eqref{temp:orbit sum form on level l ker}) 
$$\widetilde{\phi_s\otimes\nu_1}|_{(G_r^l)^F}\cong \widetilde{\phi_s\otimes\nu_2}|_{(G_r^l)^F}.$$
So both $\widetilde{\phi_s\otimes\nu_1}$ and $\widetilde{\phi_s\otimes\nu_2}$ are extensions of $\widetilde{\phi_s\otimes\nu_1}|_{(G_r^l)^F}$. Thus by Clifford theory \cite[(6.11)~Theorem~(b)]{Isaacs_CharThy_Book} the isomorphism $J_r(s,\nu_1)\cong J_r(s,\nu_2)$ tells that $\widetilde{\phi_s\otimes\nu_1}\cong \widetilde{\phi_s\otimes\nu_2}$, which implies $\nu_1\cong \nu_2$.

\vspace{2mm}
\noindent (C) This is a direct consequence of (A), (B), and \cite[Theorem~2.13]{Hill_1993_Jordan}.

\vspace{2mm}
\noindent (D) This is immediate from the construction of $J_{r}(s,\nu)$.
\end{proof}

Now consider the graded $\mathbb{Z}$-module
$$A(\mathcal{O}_r):=\bigoplus_{n\geq0}\mathbb{Z}\mathrm{Irr}(\mathrm{GL}_n(\mathcal{O}_r)),$$ 
where for $n=0$ we let $\mathrm{Irr}(\mathrm{GL}_n(\mathcal{O}_r))=\{1\}$.
One can define on $A(\mathcal{O}_r)$ a unital commutative algebra structure: For $J_r(s_1,\nu_1)\in\mathrm{Irr}(\mathrm{GL}_{n_1}(\mathcal{O}_r))$ and $J_r(s_2,\nu_2)\in\mathrm{Irr}(\mathrm{GL}_{n_2}(\mathcal{O}_r))$, put
\begin{equation}\label{temp:multiplication m}
m(J_r(s_1,\nu_1), J_r(s_2,\nu_2))
:=
\mathrm{Ind}_{L_{(s_1,s_2)}^F{\mathrm{GL}_{n_1+n_2}(\mathcal{O}_r)}^l}^{\mathrm{GL}_{n_1+n_2}(\mathcal{O}_r)}\widetilde{\boxtimes_{i}\phi_{s_i}\otimes\nu_i}
=J_r((s_1,s_2),\nu_1\boxtimes\nu_2)
\end{equation}
if the characteristic polynomials of $s_1$ and of $s_2$ are coprime, where $L_{(s_1,s_2)}$ denotes the centraliser of $\rho_r(\mathrm{diag}(A_{s_1},A_{s_2}))\in\mathrm{GL}_{n_1+n_2}(\mathcal{O}^{\mathrm{ur}}_r)$ (note that $L_{(s_1,s_2)}\cong L_{s_1}\times L_{s_2}$ and $L_{(s_1,s_2)}^F\cong L_{s_1}^F\times L_{s_2}^F$ by the coprimality between the characteristic polynomials); otherwise 
let  
$$
m(J_r(s_1,\nu_1), J_r(s_2,\nu_2))=0.
$$
Extending $m(-,-)$ linearly and setting $m(1,a)=m(a,1)=a$ for every $a\in A(\mathcal{O}_r)$, we get on $A(\mathcal{O}_r)$ the desired (associative and commutative) algebra structure.

\begin{defi}\label{defi:cuboidal rep}
Call $J_r(s,\nu)\in\mathrm{Irr}(\mathrm{GL}_{n}(\mathcal{O}_r))$  \emph{cuboidal},  if the characteristic polynomial $E_s(x)$ of $s$ is a power of an irreducible polynomial over $\mathbb{F}_q$. Denote by $\square_n$ the set of (equivalence classes of) cuboidal representations of $\mathrm{GL}_n(\mathcal{O}_r)$.
\end{defi}

\begin{expl}\label{expl:prime n}
In this example let $n$ be a prime. Then the set $\square_n$ can be partitioned into two subclasses:
\begin{equation*}
\square_n=\left\{ J_r(s,\nu) \mid E_s(x)\ \textrm{is irreducible} \right\} \bigsqcup \left\{ J_r(s,\nu) \mid \exists a\in\mathbb{F}_q\ \textrm{s.t.}\ E_s(x)=(x-a)^n  \right\}.
\end{equation*} 
The elements in the first set are called strongly cuspidal in literature and have been studied in \cite{Aubert_Onn_Prasad_Stasinski_Israelpaper_2010}, and the elements in the second set are linear twists of nilpotent orbit representations.
\end{expl}

\begin{lemm}\label{lemm:dual of Jordan}
One has $J_r(s,\nu)^* = J_r(-s,(\phi_s\phi_{-s})^{-1}\otimes\nu^*)$,
where $(-)^*$ means taking the dual.
\end{lemm}

\begin{proof}
This follows from the construction of $J_r(s,\nu)$.
\end{proof}

The construction of $\phi_s$ depends on the choices of the group morphism $\psi$, the lifts $\hat{s}_i$, and the irreducible constituents $\phi''_{s_i}$. For $p\neq 2$ we can normalise the process so that one has
\begin{equation}\label{temp: normalisation of phi_s}
\phi_{-s}=\phi_s^{-1}
\end{equation}
for every semisimple $s\in M_n(\mathbb{F}_q)$; however, this is not always doable for $p=2$, as can be seen in the case $\mathcal{O}_r=\mathbb{Z}/4$: If $\psi$ is non-trivial on the reduction map kernel $2\mathbb{Z}/4$, then it must be of order $4$. In any case, if $p\neq2$ in Lemma~\ref{lemm:dual of Jordan} we would have $J_r(s,\nu)^* = J_r(-s,\nu^*)$ once such a normalisation \eqref{temp: normalisation of phi_s} is taken.

\begin{prop}\label{prop:cuboidal rep}
The algebra $A(\mathcal{O}_r)$ is generated by
$\cup_{n\geq1}\square_n$ over $\mathbb{Z}$, and it is involutive with respect to the involution of taking dual $(-)^*$.
\end{prop}

\begin{proof}
The first assertion follows from Theorem~\ref{thm:main}(C) and the definition of $m$, and the second assertion follows from Lemma~\ref{lemm:dual of Jordan}.
\end{proof}

The above proposition suggests that the cuboidal representations form the building blocks of the representation theory of $\mathrm{GL}_n(\mathcal{O}_r)$, similar to the role played by the cuspidal representations for $\mathrm{GL}_n(\mathbb{F}_q)$; see also Proposition~\ref{prop:role of cuboidal in cohom rep} for its use in Lusztig inductions.

\vspace{2mm} Besides $m$, there is also the ``co-'' version
$$m^*\colon  A(\mathcal{O}_r)\longrightarrow 
A(\mathcal{O}_r)\otimes A(\mathcal{O}_r)
=
\bigoplus_{n\geq0}\bigoplus_{\substack{a,b\geq0;\\ a+b=n }}  \mathbb{Z}\mathrm{Irr}(\mathrm{GL}_a(\mathcal{O}_r)\times \mathrm{GL}_b(\mathcal{O}_r))$$
obtained by linearly extending
$$m^*\colon J_r(s,\nu)\longmapsto \sum_{E_{s_1}, E_{s_2}} J_r(s_1,\nu_1)\otimes J_r(s_2,\nu_2) $$
in which the sum in the RHS runs over all the \emph{coprime} factorisations $E_s(x)=E_{s_1}(x)E_{s_2}(x)$ over $\mathbb{F}_q$ (i.e.\ $E_{s_1}(x)$ and $E_{s_2}(x)$ are coprime polynomials in $\mathbb{F}_q[x]$); here we regard $E_{s_1}(x)E_{s_2}(x)$ and $E_{s_2}(x)E_{s_1}(x)$ as two different factorisations, and we require $m^*(1)=1\otimes 1$. Note that  each such factorisation gives rise, up to conjugation, a unique pair of pairs $((s_1,\nu_1),(s_2,\nu_2))$, and hence $m^*$ is co-associative and co-commutative.

\vspace{2mm} Let $\langle -, - \rangle_{A(\mathcal{O}_r)}$ (resp.\ $\langle -, - \rangle_{A(\mathcal{O}_r)\otimes A(\mathcal{O}_r)}$) be the $\mathbb{Z}$-valued bilinear form on $A(\mathcal{O}_r)$ (resp.\ $A(\mathcal{O}_r)\bigotimes A(\mathcal{O}_r)$) determined by the property that the $\mathbb{Z}$-basis $\bigsqcup_{n\geq0} \mathrm{Irr}(\mathrm{GL}_n(\mathcal{O}_r))$  (resp.\ $\bigsqcup_{a,b\in\mathbb{Z}_{\geq0}} \mathrm{Irr}(\mathrm{GL}_a(\mathcal{O}_r))\times \mathrm{Irr}(\mathrm{GL}_b(\mathcal{O}_r))$) is orthonormal. 

\begin{prop}\label{prop:adjunction and primitivity}
The pair $(m,m^*)$ is adjoint in the sense that
$$\langle m(\rho_1,\rho_2), \rho\rangle_{A(\mathcal{O}_r)}
= 
\langle \rho_1\otimes\rho_2, m^*(\rho) \rangle_{A(\mathcal{O}_r)\otimes A(\mathcal{O}_r)}$$
for all $\rho_1,\rho_2,\rho\in A(\mathcal{O}_r)$.
Moreover, an irreducible representation of $\mathrm{GL}_n(\mathcal{O}_r)$ is cuboidal if and only if $m^*(\rho)=\rho\otimes1+1\otimes\rho$.
\end{prop}

\begin{proof}
By linearity one can assume that $\rho_1,\rho_2,\rho$ are irreducible representations. Then both of the assertions follow from direct computations.
\end{proof}

\vspace{2mm} In \cite{Zelevinsky_1981_bk} Zelevinsky studied a delicate Hopf algebra structure on  $\bigoplus_{n\geq0}\mathbb{Z}\mathrm{Irr}(\mathrm{GL}_n(\mathbb{F}_q))$, whose multiplication $\tilde{m}$ and co-multiplication $\tilde{m}^*$ are defined using the Harish-Chandra induction/restriction, in which a basic observation (\cite[9.2]{Zelevinsky_1981_bk}) is that the cuspidal representations of $\mathrm{GL}_n(\mathbb{F}_q)$ are precisely the irreducible primitive elements of $\bigoplus_{n\geq0}\mathbb{Z}[\mathrm{Irr}(\mathrm{GL}_n(\mathbb{F}_q))]$, namely the irreducible elements $\rho$ satisfying that
$$\tilde{m}^*(\rho)=\rho\otimes1+1\otimes\rho,$$
a property similar to the second assertion of Proposition~\ref{prop:adjunction and primitivity} for cuboidal representations. Since then, several generalisations and variations have been made; see \cite{Crisp_Meir_Onn_varHarishChandra}, \cite{Crisp_Meir_Onn_inductive_approach_GLn}, \cite{Chen_2024_Hopfalg_duality}, and the references therein. 

\vspace{2mm} In our setting, if $u\colon \mathbb{Z}=\mathbb{Z}\mathrm{Irr}(\mathrm{GL_0}(\mathcal{O}_r))\hookrightarrow A(\mathcal{O}_r)$ denotes the natural embedding (the unit) and  $u^*\colon A(\mathcal{O}_r)\twoheadrightarrow \mathbb{Z}\mathrm{Irr}(\mathrm{GL_0}(\mathcal{O}_r))$ denotes the natural quotient (the co-unit), then one may expect the algebra/co-algebra $(A(\mathcal{O}_r),m,m^*,u,u^*)$ is a Hopf algebra. Unfortunately, this property is not true: For every $\rho\in\square_n$ one has 
$$m^*(m(\rho,\rho))=0\neq 2\rho\otimes\rho=(1\otimes\rho+\rho\otimes1)\cdot(1\otimes\rho+\rho\otimes1)=m^*(\rho)\cdot m^*(\rho),$$ 
so the Hopf axiom (i.e.\ the property asserting that $m^*$ is a ring morphism) fails in general. Nevertheless, it holds for the coprime pairs:

\begin{prop}\label{prop:Hopf axiom}
Let $\rho_i=J_r(s_i,\nu_i)\in A(\mathcal{O}_r)$ ($i=1,2$) be two irreducible representations such that $E_{s_1}$ and $E_{s_2}$ are coprime. Then
$$m^*(m(\rho_1,\rho_2))=m^*(\rho_1)\cdot m^*(\rho_2),$$ 
that is, the Hopf axiom holds for $(\rho_1, \rho_2)$.
\end{prop}

\begin{proof}
Let $E_s(x)=\prod_jp_j(x)$ be the primary decomposition over $\mathbb{F}_q$ (so the $p_j$'s are pairwise coprime and each of them is a power of an irreducible polynomial). Let $P_s$ be the set of all these $p_j$'s. Then by definition the components of $m^*(J_r(s,\nu))$ are in a natural bijection with the ordered $2$-partitions of $P_s$ (we allow the empty subset); in particular, the number of components of $m^*(J_r(s,\nu))$ is $2^{|P_s|}$. Since $E_{s_1}$ and $E_{s_2}$ are coprime, we have
$$2^{|P_{s_1}\cup P_{s_2}|}=2^{|P_{s_1}|}\cdot 2^{|P_{s_2}|},$$
from which the desired equality follows.
\end{proof}

\section{Representations from orbits}\label{sec:rep from orbits}

In this section we assume that $\mathcal{O}=\mathbb{F}_q[[\pi]]$ and $r=2l$ is even. In the below we construct a map
$$
\mathcal{J}_r \colon \mathfrak{gl}_n(\mathbb{F}_q)\longrightarrow \mathrm{Irr}(G_r^{F})
$$ 
satisfying that $\omega\in \Omega(\mathcal{J}_r(\omega))$ for every $\omega\in\mathfrak{gl}_n(\mathbb{F}_q)$. This is a simple application of the Jordan induction.

\begin{lemm}\label{lemm:nilp centraliser nilp}
Let $N$ be a nilpotent matrix in $M_n(\overline{\mathbb{F}}_q)\subseteq M_n(\mathcal{O}^{\mathrm{ur}}_m)$. Then $\mathrm{Tr}(NX)=0$ for any $X\in C_{M_n(\mathcal{O}^{\mathrm{ur}}_m)}(N)$ and $m\in\mathbb{Z}_{>0}$.
\end{lemm}
\begin{proof}
Write $X=\sum_{i=0}^{m-1}\pi^{i}X_i$ with $X_i\in M_n(\overline{\mathbb{F}}_q)$. Then $XN=NX$ implies that $X_iN=NX_i$ for $\forall i$. So $(NX_i)^n=N^nX_i^n=0$ for $\forall i$, which means that every $NX_i\in M_n(\overline{\mathbb{F}}_q)$ is a nilpotent matrix. Thus 
$$\mathrm{Tr}(NX)=\sum_{i=0}^{m-1}\pi^i\mathrm{Tr}(NX_i)=0,$$
as desired.
\end{proof}

From now on, for $N\in M_n(\mathbb{F}_q)$ we write $C(N)$ for its centraliser in $G_r^F$. Let $\psi_N$ be as defined in \eqref{temp:char of Lie alg} (by viewing $N$ as a matrix in $M_n(\mathcal{O}_l)=M_n(\mathbb{F}_q[[\pi]]/\pi^l)$). Since $\mathrm{Tr}(-)$ is invariant under conjugation, we have $\mathrm{Stab}_{G_r^F}(\psi_N)=C(N)(G_r^{l})^F$.

\begin{lemm}\label{lemm:nilp char trivial extension}
Let $N$ be a nilpotent matrix in $M_n({\mathbb{F}}_q)$. Then the assignment
$$\tilde{\psi}_N\colon x\cdot g \longmapsto \psi_N(g)$$
for $x\in C(N)$ and $g\in(G_r^{l})^F$ is a well-defined group morphism from $\mathrm{Stab}_{G_r^F}(\psi_N)$ to $\overline{\mathbb{Q}}_{\ell}^{\times}$, and hence is an extension of $\psi_N$.
\end{lemm}
\begin{proof}
If $x\in C(N)\cap G_r^{l}$, then $\psi_N(x)=1$ by Lemma~\ref{lemm:nilp centraliser nilp} (applied to $m=l$). So, as a map $\tilde{\psi}_{N}$ is well-defined; it is a group morphism because $C(N)$ centralises $N$.
\end{proof}

\begin{prop}\label{prop:nilp orbit rep construction}
Let $N$ be a nilpotent matrix in $M_n({\mathbb{F}}_q)$. Then $$\nu_N:=\mathrm{Ind}_{C(N)(G_r^{l})^F}^{G_r^F}\tilde{\psi}_N$$
is an irreducible representation of $G_r^F$ with $N\in \Omega(\nu_N)$.
\end{prop}
\begin{proof}
The irreducibility follows from Lemma~\ref{lemm:nilp char trivial extension} and Clifford theory \cite[(6.11)~Theorem~(a)]{Isaacs_CharThy_Book}. The orbit assertion is immediate from the Mackey intertwining formula (by taking restriction to $(G_r^{r-1})^F$).
\end{proof}

Now for general $\omega\in M_n(\mathbb{F}_q)$, let $\omega=s+N$ be the additive Jordan decomposition with $s$ the semisimple part and $N$ the nilpotent part. So, by \eqref{temp: defn of rational centraliser} some $\mathrm{GL}_n(\mathbb{F}_q)$-conjugate of $N$ can be written as $\mathrm{diag}(N_1,...,N_d)$, where each $N_i$ is a nilpotent matrix in $M_{n_i}(\mathbb{F}_{q^{\lambda_i}})$. 

\begin{coro}\label{coro:from orbit to irrep}
Let $\omega\in M_n(\mathbb{F}_q)$ and let the notation be as above. Then
$$\mathcal{J}_r(\omega):=J_r(s,\boxtimes_i\nu_{N_i})$$ 
is an irreducible representation of $G_r^F$ with $\omega\in\Omega(\mathcal{J}_r(\omega))$. 
\end{coro}
\begin{proof}
This follows from Theorem~\ref{thm:main} and Proposition~\ref{prop:nilp orbit rep construction}.
\end{proof}

\begin{remark}\label{remark:from orbit to irrep: char=0 and r=2}
For $\mathrm{char}(\mathcal{O})=0$, Lemma~\ref{lemm:nilp centraliser nilp} is not true in general: Take $m=2$ and consider the nilpotent matrix
$N=
\begin{pmatrix}
    1 & 1 \\ 
    1 & 1
\end{pmatrix}
\in M_2(\mathbb{F}_2)$; we also write $N$ for the canonical lift in $M_2(\mathbb{Z}/4)$. Then 
$NI_2=I_2N$ but $\mathrm{Tr}(N\cdot I_2)=2\neq0\in\mathbb{Z}/4$. However, if $r=2$, then Lemma~\ref{lemm:nilp char trivial extension}, Proposition~\ref{prop:nilp orbit rep construction}, and Corollary~\ref{coro:from orbit to irrep} are still true for $\mathrm{char}(\mathcal{O})=0$, because for $r=2$ one only need the $m=1$ case in Lemma~\ref{lemm:nilp centraliser nilp}.
\end{remark}

\section{Relations with cohomology}\label{sec:cohom}

Let $T\subseteq L_s$ denote a conjugate of (the Greenberg functor image of) a maximal torus of ${\mathrm{GL}_n}_{/\mathcal{O}^{\mathrm{ur}_r}}$ such that $FT=T$. Choose a Borel subgroup and a parabolic subgroup of ${\mathrm{GL}_n}_{/\mathcal{O}^{\mathrm{ur}_r}}$, with Greenberg functor images denoted, respectively, by $B$ and $P_s$, satisfying that (i) $T\subseteq B\subseteq P_s$, and (ii) $L_s$ is the Greenberg functor image of a Levi factor of the parabolic subgroup. Let $U,U_s$ be the Greenberg functor images of the corresponding unipotent radicals. Note that if $\alpha$ is a root (with respect to $T$) such that $\alpha(s)\neq1$, then for  $U_{\alpha}$ (the Greenberg functor image of) the root subgroup   one has $U_{\alpha}\cap L_s=\{1\}$ (see e.g.\ \cite[Remark~3.11]{ChenStasinski_2023_algebraisation_II}). 

\vspace{2mm} Consider the subgroup
$$U_s^{\pm}:=\prod_{\alpha}U_{\alpha}^l$$
of $G_r$
where the product is taken over all roots $\alpha$ such that $\alpha(s)\neq1$.  Note that this construction works for any connected reductive group, not only $\mathrm{GL}_n$, and it generalises the notion of arithmetical radical $U^{\pm}$ in \cite{ChenStasinski_2016_algebraisation}.   By definition $U_s^{\pm}$ is normalised by $L_s$ and $F$-stable. In particular, the variety $L^{-1}(U_s^{\pm})$ admits a left $G_r^F$-action and a right $L_s^F$-action, where $L\colon g\mapsto g^{-1}F(g)$ denotes the Lang map. So
\begin{equation*}
    H_c^{i}(L^{-1}(U_s^{\pm}),\overline{\mathbb{Q}}_{\ell})
\end{equation*}
admits a natural structure of $G_r^F$-module-$L_s^F$ (see \cite[Definition~5.1.1]{DM_book_2nd_edition}), where $H_c^{i}(-,\overline{\mathbb{Q}}_{\ell})$ denotes the compactly supported $\ell$-adic cohomology with $\ell$  a prime not equal to $p$. There is a natural generalisation of \cite[Proposition~3.3]{ChenStasinski_2016_algebraisation}:

\begin{prop}
Suppose that $r$ is even. Then for any $\rho\in\mathrm{Irr}(L_s^F)$ one has
$$\mathrm{Ind}_{(L_sG_r^{l})^F}^{\mathrm{GL}_n(\mathcal{O}_r)}
\widetilde{\rho}
\cong 
\sum_{i} (-1)^i H_c^{i}(L^{-1}(U_s^{\pm}),\overline{\mathbb{Q}}_{\ell})\otimes_{\overline{\mathbb{Q}}_{\ell}[L_s^F]} \rho,$$
where  $\widetilde{\rho}$ denotes the trivial inflation along the natural projection $(L_sG_r^{l})^F\rightarrow L_s^F$.
\end{prop}

\begin{proof}
    The exact same argument as in \cite[Proposition~3.3]{ChenStasinski_2016_algebraisation} works here,  with $U^{\pm}$ formally replaced by $U_s^{\pm}$.
\end{proof}

In particular,   $J_r(s,\nu)$ admits a geometric realisation:

\begin{coro}\label{coro:geom of Jordan}
If $r$ is even, $J_r(s,\nu)\cong \sum_{i} (-1)^i H_c^{i}(L^{-1}(U_s^{\pm}),\overline{\mathbb{Q}}_{\ell})\otimes_{\overline{\mathbb{Q}}_{\ell}[L_s^F]}{(\phi_s\otimes\nu)}$.
\end{coro}

\vspace{2mm} Similar to $H_c^{i}(L^{-1}(U_s^{\pm}),\overline{\mathbb{Q}}_{\ell})$, one can consider the $G_r^F$-module-$L_s^F$
$$H_c^{i}(L^{-1}(F(U_s)),\overline{\mathbb{Q}}_{\ell});$$
it gives rise to the higher analogue of the  Lusztig induction functor  \cite{Lusztig_1976_finiteness_unipotent_classes}
$$R_{L_s,U_s}^{\mathrm{GL}_n}\colon \rho\longmapsto \sum_i (-1)^i H_c^{i}(L^{-1}(F(U_s)),\overline{\mathbb{Q}}_{\ell})\otimes_{\overline{\mathbb{Q}}_{\ell}[L_s^F]}\rho$$
which takes virtual $L_s^F$-modules to virtual $G_r^F$-modules. This construction was introduced by Charlotte Chan  \cite{Chan_2024_Scalar_Product}  in the general parahoric group setting, via a character-theoretic property instead of the above bimodule tensor product (the two constructions coincide by \cite[Proposition~5.1.5]{DM_book_2nd_edition}).

\vspace{2mm} The following conjecture is an exact generalisation (in the case of $\mathrm{GL}_n$) of the algebraisation theorem in \cite{ChenStasinski_2016_algebraisation,   ChenStasinski_2023_algebraisation_II}:

\begin{conj}\label{conj:Jordan=Lusztig}
Suppose that $r$ is even. Then for every $(s,\nu)$ one has
$$J_r(s,\nu)
\cong R_{L_s,U_s}^{\mathrm{GL}_n}(\phi_s\otimes\nu);$$
in particular, $R_{L_s,U_s}^{\mathrm{GL}_n}(\phi_s\otimes\nu)$ is irreducible.
\end{conj}

Note that, together with Theorem~\ref{thm:main}, this conjecture gives a perfect analogue of the relation between Lusztig induction and Lusztig's Jordan decomposition; see \cite[Theorem~11.4.3(ii)]{DM_book_2nd_edition}.

\vspace{2mm} When $\Omega(\nu)=\{ 0 \}$, the representation $J_r(s,\nu)$ is a semisimple orbit representation. In this case a similar conjecture (using Yu's representations \cite{Yu_2001_JAMS} and Fintzen--Kaletha--Spice's twists \cite{Fintzen_Kaletha_Spice_twistedYu_DMJ_2023}) has been proposed by Chan--Oi in \cite[Conjecture~5.12]{Chan_Oi_2025_GreenFunc} for parahoric groups, under some assumption on $p$. A smaller family of them, called  strongly semisimple representations, has been constructed previously by Hill \cite{Hill_1995_semisimple}; see also the discussion in \cite[Subsection~8.2]{Aubert_Onn_Prasad_Stasinski_Israelpaper_2010}. Recently, Monteiro \cite{monteiro2023stable} introduced the notion of stable matrices and used it to construct another class of irreducible representations of $\mathrm{GL}_n(\mathcal{O}_r)$; Monteiro's construction covers all the strongly semisimple orbit representations.

\begin{prop}\label{prop:simplest cases of J=L conj}
Conjecture~\ref{conj:Jordan=Lusztig} is true in the following situations:
\begin{enumerate}
\item[(i)] $s$ is regular (i.e., the characteristic polynomial of $s\in M_n(\mathbb{F}_q)$ is separable over $\overline{\mathbb{F}}_q$);
\item[(ii)] $s$ splits (i.e., the characteristic polynomial of $s\in M_n(\mathbb{F}_q)$ splits over ${\mathbb{F}}_q$);
\item[(iii)] $n=2$ or $3$.
\end{enumerate}
\end{prop}

Note that if $n$ is a prime, then (i) and (ii) cover all the cuboidal representations of $\mathrm{GL}_n(\mathcal{O}_r)$;  see Example~\ref{expl:prime n} and Proposition~\ref{prop:role of cuboidal in cohom rep}.

\begin{proof} (i) This is a special case of \cite[Theorem~4.4]{ChenStasinski_2023_algebraisation_II} (combined with \cite[Proposition~2.2 and Remark~4.6]{ChenStasinski_2023_algebraisation_II}); see also \cite{ChenStasinski_2016_algebraisation}.

\vspace{2mm} (ii) In this case, up to conjugation we can assume that $P_s$ (and hence $U_s$) is $F$-stable.
Then the same argument as in \cite[Page~139]{DM_book_2nd_edition} gives that
    $$R_{L_s,U_s}^{\mathrm{GL}_n}(\phi_s\otimes\nu)
    \cong{\mathrm{Ind}}_{P_s^F}^{G_r^F}\widetilde{\phi_s\otimes\nu}.$$
    By construction, both $J_r(s,\nu)$ and ${\mathrm{Ind}}_{P_s^F}^{G_r^F}\widetilde{\phi_s\otimes\nu}$ have the same dimension, so, as $J_r(s,\nu)$ is irreducible (Theorem~\ref{thm:main}~(A)), it suffices to show that their inner product is non-zero.

\vspace{2mm} By Frobenius reciprocity and Mackey's intertwining formula we see that
\begin{equation*}
    \begin{split}
   \langle J_r(s,\nu),\ & {\mathrm{Ind}}_{P_s^F}^{G_r^F} \widetilde{\phi_s\otimes\nu}  \rangle_{G_r^F} \\
    & = \langle \mathrm{Res}_{P_s^F}^{G_r^F} J_r(s,\nu), \widetilde{\phi_s\otimes\nu}\rangle_{P_s^F}\\
    & = \sum_{x\in (L_sG_r^l)^F\backslash G_r^F/P_s^F} 
    \langle \mathrm{Ind}^{P_s^F}_{{{}^x(L_sG_r^l)}^F\cap P_s^F} \mathrm{Res}^{{{}^x(L_sG_r^l)}^F}_{{{}^x(L_sG_r^l)}^F\cap P_s^F} {^x(\widetilde{\phi_s\otimes\nu})}, \widetilde{\phi_s\otimes\nu}    \rangle_{P_s^F}\\
    & = \sum_{x\in (L_sG_r^l)^F\backslash G_r^F/P_s^F} 
    \langle  \mathrm{Res}^{{{}^x(L_sG_r^l)}^F}_{{{}^x(L_sG_r^l)}^F\cap P_s^F} {^x(\widetilde{\phi_s\otimes\nu})},  \mathrm{Res}^{P_s^F}_{{{}^x(L_sG_r^l)}^F\cap P_s^F} \widetilde{\phi_s\otimes\nu}    \rangle_{{{}^x(L_sG_r^l)}^F\cap P_s^F};
    \end{split}
\end{equation*}
clearly the summand corresponding to $x=1$ is non-zero, and so the whole inner product is non-zero, as desired.

\vspace{2mm} (iii) This case is covered by (i) and (ii).
\end{proof}

\begin{prop}\label{prop:role of cuboidal in cohom rep}
If Conjecture~\ref{conj:Jordan=Lusztig} holds true for all the cuboidal $J_r(s,\nu)\in\mathrm{Irr}(\mathrm{GL}_m(\mathcal{O}_r))$ for all $m>0$, then it holds true in full.
\end{prop}
\begin{proof}
We prove this by an induction on $n$ in $\mathrm{GL}_n$. The initial step is covered by Proposition~\ref{prop:simplest cases of J=L conj}(iii). Suppose now that the conjecture is true for any $k<n$. If $J_r(s,\nu)\in\mathrm{Irr}(\mathrm{GL}_n(\mathcal{O}_r))$ is not cuboidal, then in the same way as \eqref{temp:multiplication m} we can rewrite $\phi_s\otimes\nu$ to be $(\phi_{s_1}\otimes\nu_1)\boxtimes(\phi_{s_2}\otimes\nu_2)$, where $n_1+n_2=n$ and each $\phi_{s_i}\otimes\nu_i$ is an irreducible representation of $L_{s_i}^F\subseteq \mathrm{GL}_{n_i}(\mathcal{O}_r)$. Let $P_{n_1,n_2}\subseteq \mathrm{GL}_n(\mathcal{O}^{\mathrm{ur}}_r)$ be the corresponding standard blocked upper triangular subgroup, and denote its blocked upper uni-triangular part by $U_{n_1,n_2}$. Then by the transitive property of Lusztig inductions (see \cite[Proposition~3.3]{Chan_2024_Scalar_Product}) we have
\begin{equation*}
\begin{split}
    R_{L_s,U_s}^{\mathrm{GL}_n}(\phi_s\otimes\nu)
    &=R_{\mathrm{GL}_{n_1}\times \mathrm{GL}_{n_2},U_{n_1,n_2}}^{\mathrm{GL}_n}
    R^{\mathrm{GL}_{n_1}\times \mathrm{GL}_{n_2}}_{L_s,U_s\cap(\mathrm{GL}_{n_1}\times \mathrm{GL}_{n_2})}
    (\phi_{s_1}\otimes\nu_1)\boxtimes(\phi_{s_2}\otimes\nu_2)\\
    &=\mathrm{Ind}_{P_{n_1,n_2}(\mathcal{O}_r)}^{\mathrm{GL}_n(\mathcal{O}_r)} 
    \widetilde{ J_r(s_1,\nu_1)\boxtimes J_r(s_2,\nu_2) },
\end{split}
\end{equation*}
in which the second equality follows from the induction principle. Now the conclusion follows from the same use of a dimension comparison and of the Mackey formula as in the argument of Proposition~\ref{prop:simplest cases of J=L conj}(ii).
\end{proof}

\section{On odd levels}\label{sec:odd levels}

In this section we propose a conjectural construction of  Jordan induction $J_r(s,\nu)$ in the odd level case.

\begin{lemm}
Suppose that $r=2l-1$ is odd. Let  $\nu$ be a nilpotent orbit irreducible representation of $L_s^F$. Then all the irreducible constituents of 
$$\mathrm{Ind}_{\left(L_s^1G_{r}^{l}\right)^F}^{\left(L_s^1G_{r}^{l'}\right)^F}\widetilde{(\phi_s\otimes\nu)}|_{\left(L_s^1G_{r}^l\right)^F}$$
have the same dimension and the same multiplicity. 
\end{lemm}

\begin{proof}
Let $\rho$ be an irreducible constituent of $\mathrm{Ind}_{\left(L_s^1G_{r}^{l}\right)^F}^{\left(L_s^1G_{r}^{l'}\right)^F}\widetilde{(\phi_s\otimes\nu)}|_{\left(L_s^1G_{r}^l\right)^F}$. By Clifford theory (see \cite[(6.11)~Theorem~(b)]{Isaacs_CharThy_Book}) we know that for
$$S:=\mathrm{Stab}_{\left(L_s^1G_{r}^{l'}\right)^F}\left(   \widetilde{(\phi_s\otimes\nu)}|_{\left(L_s^1G_{r}^l\right)^F}  \right)$$
the induced representation $\mathrm{Ind}^S_{\left(L_s^1G_{r}^{l}\right)^F}   \widetilde{(\phi_s\otimes\nu)}|_{\left(L_s^1G_{r}^l\right)^F}$ has an irreducible constituent $\rho_S$ satisfying that
$$\rho\cong \mathrm{Ind}_S^{\left(L_s^1G_{r}^{l'}\right)^F}\rho_S.$$
Meanwhile, note that
$$S/(L_s^1G^l_r)^F\subseteq (L_s^1G^{l'}_r)^F/(L_s^1G^l_r)^F=(G^{l'}_r)^F/(L_s^{l'}G^l_r)^F$$
is a quotient of $(G^{l'}_r)^F/(G^l_r)^F=(G_l^{l'})^F$, and hence it  
is abelian. So by \cite[Corollary~1.24(b)]{Navarro_2018_book_McKayConj} every irreducible constituent of $\mathrm{Ind}^S_{\left(L_s^1G_{r}^{l}\right)^F}   \widetilde{(\phi_s\otimes\nu)}|_{\left(L_s^1G_{r}^l\right)^F}$  is of the form 
$\lambda\rho_S$
where $\lambda$ is (the trivial lift of) a $1$-dimensional representation of the abelian group $S/(L_s^1G^l_r)^F$.

\vspace{2mm} In particular, every irreducible constituent of $\mathrm{Ind}_{\left(L_s^1G_{r}^{l}\right)^F}^{\left(L_s^1G_{r}^{l'}\right)^F}\widetilde{(\phi_s\otimes\nu)}|_{\left(L_s^1G_{r}^l\right)^F}$ is 
$$\mathrm{Ind}_S^{\left(L_s^1G_{r}^{l'}\right)^F}\lambda \rho_S,$$
for some $1$-dimensional $\lambda$. This proves the dimension assertion.

\vspace{2mm} Now, since $\lambda$ is by definition trivial on $(L_s^1G_r^l)^F$, the multiplicity assertion   follows immediately  from Clifford theory (\cite[(6.11)~Theorem~(d)]{Isaacs_CharThy_Book}) and the Frobenius reciprocity.
\end{proof}

We shall fix for each pair $(s,\nu)$ an irreducible constituent $\phi_{s,\nu}$   of the  induced representation $\mathrm{Ind}_{\left(L_s^1G_{r}^{l}\right)^F}^{\left(L_s^1G_{r}^{l'}\right)^F}\widetilde{(\phi_s\otimes\nu)}|_{\left(L_s^1G_{r}^l\right)^F}$ in the above. We make the following conjecture concerning its construction and extension:

\begin{conj}\label{conj:odd levels}
Suppose that $r=2l-1$ is odd. Then:
\begin{enumerate}
\item[(i)] One has
$$\mathrm{Ind}_{\left(L_s^1G_{r}^{l}\right)^F}^{\left(L_s^1G_{r}^{l'}\right)^F}\widetilde{(\phi_s\otimes\nu)}|_{\left(L_s^1G_{r}^l\right)^F}
=
\sqrt{\frac{|\left(L_s^1G_{r}^{l'}\right)^F|}{|\left(L_s^1G_{r}^{l}\right)^F|}} \cdot \phi_{s,\nu};$$
    
\item[(ii)] $\phi_{s,\nu}$ admits a canonical extension $\widehat{\phi_{s,\nu}}$ to $(L_sG_r^{l'})^F$, in a way such that, for a given $s$, non-isomorphic $\nu$ yields non-isomorphic $\widehat{\phi_{s,\nu}}$.
\end{enumerate}
\end{conj}

One may compare this conjecture with the constructions in \cite[Section~4]{ChenStasinski_2023_algebraisation_II}, in which it is established for regular semisimple $s$ (after specialising it to $\mathrm{GL}_n$); in particular, the $\phi_{s,\nu}$ might be viewed as a non-abelian version of the Heisenberg lift in \cite[Lemma~4.1]{ChenStasinski_2023_algebraisation_II} (see also \cite[Corollary~3.3]{Stasinski_Stevens_2016_regularRep} and \cite[(8.3.3)~Proposition]{Bushnell_Froelich_book_GaussSum_1983}).

\begin{prop}\label{prop:main-odd}
Suppose that $r=2l-1$ is odd, and assume that Conjecture~\ref{conj:odd levels} holds. Let
$$J_r(s,\nu):=\mathrm{Ind}_{(L_sG_r^{l'})^F}^{G_r^F}\widehat{\phi_{s,\nu}}.$$
Then the same assertions of Theorem~\ref{thm:main} hold, namely:
\begin{enumerate}
\item[(A)] (Irreducibility) The representations $J_r(s,\nu)$ are always irreducible.
\item[(B)] (Uniqueness) If $s,s'\in M_n(\mathbb{F}_q)$ are not $\mathrm{GL}_n(\mathbb{F}_q)$-conjugated, then 
$$\mathrm{Im}(J_r(s,-))\cap \mathrm{Im}(J_r(s',-))=\emptyset;$$
moreover, $J_r(s,-)$ is an injection for every semisimple $s \in M_n(\mathbb{F}_q)$.
\item[(C)] (Exhaustion) One has 
$$\mathrm{Irr}(G_r^F)=\mathrm{Im}(J_r(-,-)).$$
\item[(D)] (Dimension) The ratio $\frac{\dim J_r(s,\nu)}{\dim\nu}$ is a constant depending only on $s$.
\end{enumerate}
\end{prop}

\begin{proof}
The exactly same argument of Theorem~\ref{thm:main} works here.
\end{proof}

\section{An example}\label{sec:GL3 r=2}

The irreducible representations of $\mathrm{GL}_2(\mathcal{O}_r)$ have been studied by many authors, and their full classification and construction are known for decades, in various contexts  \cite{kutzko1973characters,Nagornyj_1976_GL2Zpn,Nobs_1977_GL2,Onn_AdvMath_2008,Sta2009smooth}; see \cite[Introduction]{Sta2009smooth} for the history. For $\mathrm{GL}_3$, to exhaust all the irreducible representations is a difficult problem, and the challenging part is to construct the subregular nilpotent ones; when $p>3$ this is solved by Onn--Prasad--Singla \cite[Section~7]{OnnPrasadSingla_2025_zetaA2poschar}.

\vspace{2mm} In the remaining part of this section we assume that $n=3$ and $r=2$ (we make no assumption on $p$). So we are considering
$$G_2^F=\mathrm{GL}_3(\mathcal{O}_2).$$
We want to describe the construction of its irreducible representations in the framework of Jordan induction; we shall be concerned with only the primitive representations in the following sense.

\begin{defi}\label{defi:primitive}
An irreducible representation $\sigma$ of $G_r^F$ is called \emph{primitive} if $\Omega(\sigma)$ does not consist of a central element. Otherwise it is called imprimitive.
\end{defi}

By Theorem~\ref{thm:main}(C) every irreducible representation of $G_r^F$ is of the form $J_r(s,\nu)$, and by Lemma~\ref{lemm:s.s. part of orbit} the semisimple part of every element of $\Omega(J_r(s,\nu))$ is conjugate to $s$. Hence $J_r(s,\nu)$ is primitive if and only if either $s$ is non-central, or $s$ is central and $\Omega(\nu)\neq\{0\}$. Note that a semisimple $s\in M_3(\mathbb{F}_q)$ is conjugate to exactly one of the following:
\begin{itemize}
\item[(i)] a regular semisimple element (three distinct eigenvalues in $\overline{\mathbb{F}}_q$);
\item[(ii)] an element of type $(2,1)$, i.e.\ $\mathrm{diag}(a,a,b)$ with $a\neq b\in\mathbb{F}_q$;
\item[(iii)] a central element $s=c\cdot I_3$, $c\in\mathbb{F}_q$.
\end{itemize}

\vspace{2mm} \noindent {\bf (i) The regular semisimple case}

\vspace{2mm} Suppose that $s$ is regular and semisimple. Then 
$$L_s^F\cong T(\mathcal{O}_2)$$
for a maximal torus $T$ of $\mathrm{GL}_3$ over $\mathcal{O}_2$;  in this case the nilpotent orbit representations of $L_s^F$ are precisely the trivial inflations of the irreducible representations of $(L_s)_{1}^F$. Thus:

\begin{prop}\label{prop:reg ss case}
Let $s\in M_3(\mathbb{F}_q)$ be a regular semisimple element, and let $\nu$ be an imprimitive $1$-dimensional representation of $L_s^F$. Then $J_2(s,\nu)$ is a regular semisimple orbit irreducible representation of $G_2^F$, and every regular semisimple orbit irreducible representation of $G_2^F$ is of this form. So $\dim J_2(s,\nu)=\frac{|\mathrm{GL}_3(\mathbb{F}_q)|}{|(L_s)_{1}^F|}$, and 
\begingroup \renewcommand{\arraystretch}{1.5}
\begin{center}
\begin{tabular}{l|l}
\hline
$\dim J_2(s,\nu)$ & $\#$  \\
\hline
$q^3(q^2+q+1)(q+1)$ & $\frac{1}{6}q(q-1)(q-2)\cdot (q-1)^3$  \\
$q^3(q^3-1)$ & $\frac{1}{2}q(q^2-q)\cdot(q-1)(q^2-1)$  \\
$q^3(q-1)^2(q+1)$ & $\frac{1}{3}(q^3-q)\cdot(q^3-1)$. \\
\hline
\end{tabular}
\end{center}
\endgroup
\end{prop}

\vspace{2mm}\noindent {\bf (ii) The type $(2,1)$ case}

\vspace{2mm} Suppose that $s$ is conjugate to $\mathrm{diag}(a,a,b)$ with $a\neq b\in\mathbb{F}_q$. Then 
$$L_s^F\cong \mathrm{GL}_2(\mathcal{O}_2)\times \mathcal{O}_2^{\times}.$$
An irreducible representation $\nu$ of $L_s^F$ has a nilpotent orbit if and only if it is of the form
$$\nu=\nu_2\boxtimes\chi,$$
where $\nu_2$ is a nilpotent orbit representation of $\mathrm{GL}_2(\mathcal{O}_2)$ and $\chi$ is a character of $\mathcal{O}_2^{\times}$ trivially   inflated from $\mathbb{F}_q^{\times}$. If the orbit of $\nu_2$ is non-zero, then $\nu_2$ is a regular nilpotent orbit representation of $\mathrm{GL}_2(\mathcal{O}_2)$, whose construction is given in \cite[Section~3]{Sta2009smooth}; we recall it here: Let 
$N={\begin{pmatrix}
        0 & 1\\
        0 & 0
    \end{pmatrix}}$.
Then these representations are of the form 
\begin{equation}\label{temp:regular nilp of GL_2}
\nu_2=\mathrm{Ind}_{C(N)\cdot(G_2^1)^F}^{\mathrm{GL}_2(\mathcal{O}_2)}\tilde{\psi},
\end{equation}
where 
$$C(N)=
\left\{ 
    \begin{pmatrix}
        a & b\\
        0 & a
    \end{pmatrix},
    a\in\mathcal{O}_2^{\times}, b\in \mathcal{O}_2
\right\}$$ 
and $\tilde{\psi}$ is an extension of $\psi_N(-)=\psi(\mathrm{Tr}(N\cdot a(-)))$ from $(G_2^1)^F$ to $C(N)\cdot(G_2^1)^F$; the extensions $\tilde{\psi}$, by Clifford theory, are in bijection with the linear characters of the abelian group $C(N)/(C(N)\cap G_2^1)$. In particular, those $\nu_2$ all have the dimension $$\dim\nu_2=q^2-1$$
and the number of their isomorphism classes  is
$(q-1)q$. On the other hand, if $\Omega(\nu_2)=\{ 0 \}$, then $\nu_2$ is the trivial inflation of an element in $\mathrm{Irr}(\mathrm{GL}_2(\mathbb{F}_q))$, which can be of dimension $1$, $q$, $q+1$, or $q-1$; the numbers of their isomorphism classes are $q-1$, $q-1$, $\frac{1}{2}(q-1)(q-2)$, and $\frac{1}{2}(q^2-q)$, respectively. Therefore we have:

\begin{prop}\label{prop:(2,1) case}
Every irreducible representation of $\mathrm{GL}_3(\mathcal{O}_2)$ whose orbit has a type $(2,1)$ semisimple part is of the form
$$J_2(s,\nu)=\mathrm{Ind}_{(L_sG_2^1)^F}^{\mathrm{GL}_3(\mathcal{O}_2)}
\widetilde{\phi_{s}\otimes(\nu_2\boxtimes\chi)},$$
where $s$ be of type $(2,1)$ and $\nu=\nu_2\boxtimes\chi$ is a nilpotent orbit representation of $L_s^F$ as above. 
Moreover, $\nu_2$ is either a regular nilpotent orbit irreducible representation of $\mathrm{GL}_2(\mathcal{O}_2)$, or the trivial inflation of an irreducible representation of $\mathrm{GL}_2(\mathbb{F}_q)$ (whose orbit is $\{0\}$). In the first case $\nu_2$ is given by \eqref{temp:regular nilp of GL_2} and is of dimension $q^2-1$, and
\begingroup \renewcommand{\arraystretch}{1.5}
\begin{center}
\begin{tabular}{l|l}
\hline
 $\dim J_2(s,\nu)$ & $\#$ \\
\hline
$q^{2}(q^2+q+1)\cdot(q^2-1)$ & $(q-1)q\cdot (q-1) \cdot (q-1)q$. \\
\hline
\end{tabular}
\end{center}
\endgroup
On the other hand, if $\Omega(\nu_2)=\{ 0 \}$, then
\begingroup \renewcommand{\arraystretch}{1.5}
\begin{center}
\begin{tabular}{l|l}
\hline
$\dim J_2(s,\nu)$ & $\#$ \\
\hline
$q^{2}(q^2+q+1)$ & $(q-1)\cdot (q-1) \cdot (q-1)q$ \\
$q^{2}(q^2+q+1)\cdot q$ & $(q-1)\cdot (q-1) \cdot (q-1)q$ \\
$q^{2}(q^2+q+1)\cdot (q+1)$ & $\frac{1}{2}(q-1)(q-2) \cdot (q-1) \cdot (q-1)q$ \\
$q^{2}(q^2+q+1)\cdot (q-1)$ & $\frac{1}{2}(q^2-q) \cdot (q-1) \cdot (q-1)q$. \\
\hline
\end{tabular}
\end{center}
\endgroup
\end{prop}

\vspace{2mm}\noindent {\bf (iii) The central case}

\vspace{2mm} Suppose that $s=c\cdot I_3$  ($c\in\mathbb{F}_q$). Then $L_s=G_2^F$   and  
$$J_2(s,\nu)=\phi_{c\cdot I_3}\otimes\nu,$$
where $\nu$ is a nilpotent orbit representation of $G_2^F$. Thus $J_2(s,\nu)$ is primitive if and only if $\nu$ is either regular nilpotent or subregular nilpotent.

\vspace{2mm}\noindent {\bf (iii-1) The regular nilpotent orbit}

\vspace{2mm}
Let $N:=
\begin{pmatrix}
    0 & 1 & 0 \\
    0 & 0 & 1 \\
    0 & 0 & 0
\end{pmatrix}$.
Then for 
$C(N)
:=
\left\{
    \begin{pmatrix}
    a & b & c \\
    0 & a & b \\
    0 & 0 & a
    \end{pmatrix} \mid a\in\mathcal{O}_2^{\times},\ b,c\in\mathcal{O}_2
\right\}$,
one has $\mathrm{Stab}_{G_2^F}(\psi_{N})=C(N)(G_2^1)^F$. Since $C(N)(\cong (\mathcal{O}_2[x]/x^3)^{\times})$ is abelian, by Clifford theory \cite[(6.17)~Corollary]{Isaacs_CharThy_Book} all the extensions of $\psi_N$ to $C(N)(G_2^1)^F$ are of the form $\chi \tilde{\psi}_N$, where $\tilde{\psi}_N$ is defined as in Lemma~\ref{lemm:nilp char trivial extension} (see also Remark~\ref{remark:from orbit to irrep: char=0 and r=2}) and $\chi$ is a $1$-dimensional representation $\chi$ of $C_{G_1^F}(N)\cong (\mathbb{F}_q[x]/x^3)^{\times}$. Thus, similar to \eqref{temp:regular nilp of GL_2}, the regular nilpotent orbit representations of $\mathrm{GL}_3(\mathcal{O}_2)$ are
\begin{equation}\label{temp:regular nilp of GL_3}
\nu=\mathrm{Ind}_{C(N)\cdot(G_2^1)^F}^{\mathrm{GL}_3(\mathcal{O}_2)}\chi\tilde{\psi}_N,
\end{equation}
where $\chi$ runs over the $1$-dimensional representations of $C_{G_1^F}(N)$. In summary:

\begin{prop}\label{prop:reg nilp}
Every irreducible representation of $\mathrm{GL}_3(\mathcal{O}_2)$ whose orbit has a central semisimple part and a regular nilpotent part is of the form
$$J_2(c\cdot I_3,\nu)=\phi_{c\cdot I_3}\otimes \mathrm{Ind}_{C(N)\cdot(G_2^1)^F}^{\mathrm{GL}_3(\mathcal{O}_2)}\chi\tilde{\psi}_N,$$
where $c\in\mathbb{F}_q$ and $\chi\in\mathrm{Irr}(C_{G_1^F}(N))$. We have
\begingroup \renewcommand{\arraystretch}{1.5}
\begin{center}
\begin{tabular}{l|l}
\hline
$\dim J_2(s,\nu)$ & $\#$ \\
\hline
$q(q^2-1)(q^3-1)$ & $q \cdot q^2(q-1)$. \\
\hline
\end{tabular}
\end{center}
\endgroup
\end{prop}

\vspace{2mm}\noindent {\bf (iii-2) The subregular nilpotent orbit}

\vspace{2mm} Let 
$N=\begin{pmatrix}
    0 & 1 & 0 \\
    0 & 0 & 0 \\
    0 & 0 & 0
\end{pmatrix}$.
Note that its centraliser in ${\mathrm{GL}_3(\mathcal{O}_2)}$ is
$$C(N)=
\left\{
    \begin{pmatrix}
        a & b & c\\
        0 & a & 0\\
        0 & d & e
    \end{pmatrix} \mid  a,e \in \mathcal{O}_2^{\times},\ b,c,d \in \mathcal{O}_2
\right\},$$
and its centraliser in ${\mathrm{GL}_3(\mathbb{F}_q)}$ is $\rho_{2,1}(C(N))$; so $\mathrm{Stab}_{G_2^F}(\psi_N)=C(N)\cdot(G_2^1)^F$. As in the case of (iii-1), all the subregular nilpotent orbit representations of $G_2^F$ are of the form
\begin{equation}\label{temp:subregular nilp of GL_3}
\nu=\mathrm{Ind}_{C(N)\cdot(G_2^1)^F}^{\mathrm{GL}_3(\mathcal{O}_2)}\chi\tilde{\psi}_N,
\end{equation}
where $\chi$ runs over $\mathrm{Irr}(C_{G_1^F}(N))$. So we only need to find out the irreducible representations of $C_{G_1^F}(N)$.

\vspace{2mm} Note that
$$C_{G_1^F}(N)=\rho_{2,1}(C(N))=D\ltimes H$$
where 
$D:=\{ \mathrm{diag}(a,a,e)\mid a,e\in\mathbb{F}_q^{\times} \}$
and 
$$H:=\left\{
    \begin{pmatrix}
        1 & b & c\\
        0 & 1 & 0\\
        0 & d & 1
    \end{pmatrix} \mid b,c,d \in \mathbb{F}_q
    \right\}
\cong 
\left\{
    \begin{pmatrix}
        1 & c & b\\
        0 & 1 & d\\
        0 & 0 & 1
    \end{pmatrix} \mid  b,c,d \in \mathbb{F}_q
\right\}.$$
Being isomorphic to the Heisenberg group over $\mathbb{F}_q$, the construction of irreducible representations of $H$ is well-known (details can be found in \cite[Section~1.2 and Section~1.3]{Gerardin_1977_Weil}): There are $q^2$ irreducible representations of dimension $1$ (the ones trivial on the centre) and $q-1$ irreducible representations of dimension $q$ (the ones determined by the non-trivial central characters). 

\vspace{2mm} Every $q$-dimensional $\sigma_q\in\mathrm{Irr}(H)$ is invariant under the action of $D$, as $D$ does not affect the centre of $H$. So, since $|D|$ is coprime to $|H|$, by \cite[Corollary~6.2]{Navarro_2018_book_McKayConj} we see that $\sigma_q$ extends to $C_{G_1}(N)$ in $|D|=(q-1)^2$ different ways. The trivial representation $1_H$ clearly also extends to $C_{G_1}(N)$ in $(q-1)^2$ different ways. It remains to consider the non-trivial $1$-dimensional representations $\sigma_1$ of $H$.

\vspace{2mm} Note that $\mathrm{diag}(a,a,e)\in D$ acts on $H$ by taking the entry triple $(b,c,d)$ to $(b,ae^{-1}c,ea^{-1}d)$. So the stabiliser of every $\sigma_1$ in $C_{G_1^F}(N)$ is $Z_{G_1^F}H$, and the $D$-orbit of $\sigma_1$ consists of $q-1$ elements. For the $D$-orbit of $\sigma_1$ we choose a representative $\sigma_1^o$. Since $Z_{G_1^F}$ is the centre, we can extend $\sigma_1^o$ to $Z_{G_1^F}H$, in $|Z_{G_1^F}|=q-1$ different ways. Thus by Clifford theory,
$$\mathrm{Ind}_{Z_{G_1^F}H}^{C_{G_1^F}(N)}z\sigma_1^o$$
for $z\in\mathrm{Irr}(Z_{G_1^F})$, produces $q-1$ non-isomorphic irreducible representations of dimension $q-1$. As there are totally $q+1$ orbits of non-trivial $1$-dimensional representations of $H$, this process gives $q^2-1$ irreducible representations of dimension $q-1$ of $C_{G_1^F}(N)$.

\vspace{2mm} Summing the squares of the dimensions of the irreducible representations obtained above, one sees that all irreducible representations of $C_{G_1^F}(N)$ are exhausted. So we get:

\begin{prop}\label{prop:subreg nilp}
Every irreducible representation of $\mathrm{GL}_3(\mathcal{O}_2)$ whose orbit has a central semisimple part and a subregular nilpotent part is of the form
$$J_2(c\cdot I_3,\nu)=\phi_{c\cdot I_3}\otimes \mathrm{Ind}_{C(N)\cdot(G_2^1)^F}^{\mathrm{GL}_3(\mathcal{O}_2)}\chi\tilde{\psi}_N,$$
where $c\in\mathbb{F}_q$ and $\chi\in\mathrm{Irr}(C_{G_1^F}(N))$. These $\chi$'s are in three classes: 
\begin{itemize}
\item $(q-1)^2$ extensions of $q-1$ irreducible representations $\sigma_q$ of dimension $q$ of $H$,
\item $(q-1)^2$ extensions of the trivial representation $1_H$ of $H$, and
\item $q^2-1$ induced representations $\mathrm{Ind}_{Z_{G_1^F}H}^{C_{G_1^F}(N)}z\sigma_1^o$ of dimension $q-1$.
\end{itemize}
Thus
\begingroup \renewcommand{\arraystretch}{1.5}
\begin{center}
\begin{tabular}{l|l}
\hline
$\dim J_2(s,\nu)$ & $\#$ \\
\hline
$(q+1)(q^3-1)\cdot q$ & $q(q-1)^3$ \\
$(q+1)(q^3-1)$ & $q(q-1)^2$ \\
$(q+1)(q^3-1)\cdot (q-1)$ & $q(q^2-1)$. \\
\hline
\end{tabular}
\end{center}
\endgroup
\end{prop}

\begin{remark}
Summing the squares of the dimensions (with the multiplicities) in the tables obtained in Proposition~\ref{prop:reg ss case}, Proposition~\ref{prop:(2,1) case}, Proposition~\ref{prop:reg nilp}, Proposition~\ref{prop:subreg nilp}, we get a polynomial $\delta(q):=q(q^8-1)(q^3-1)(q^3-q)(q^3-q^2)$. Thus the equality
$$
\delta(q)+q\cdot|\mathrm{GL}_3(\mathbb{F})_q|=q^9(q^3-1)(q^3-q)(q^3-q^2)=|\mathrm{GL}_3(\mathcal{O}_2)|
$$
verifies that we obtained all the primitive irreducible representations.
\end{remark}

\bibliographystyle{alpha}
\bibliography{zchenrefs}

\begin{thebibliography}{AKOV16}

\bibitem[AKOV16]{AvniKlopschOnnVoll_2016_similarity}
Nir Avni, Benjamin Klopsch, Uri Onn, and Christopher Voll.
\newblock Similarity classes of integral {$p$}-adic matrices and representation
  zeta functions of groups of type {$A_2$}.
\newblock {\em Proc. Lond. Math. Soc. (3)}, 112(2):267--350, 2016.

\bibitem[AOPS10]{Aubert_Onn_Prasad_Stasinski_Israelpaper_2010}
Anne-Marie Aubert, Uri Onn, Amritanshu Prasad, and Alexander Stasinski.
\newblock On cuspidal representations of general linear groups over discrete
  valuation rings.
\newblock {\em Israel J. Math.}, 175:391--420, 2010.

\bibitem[BF83]{Bushnell_Froelich_book_GaussSum_1983}
Colin~J. Bushnell and Albrecht Fr\"{o}hlich.
\newblock {\em Gauss sums and {$p$}-adic division algebras}, volume 987 of {\em
  Lecture Notes in Mathematics}.
\newblock Springer-Verlag, Berlin-New York, 1983.

\bibitem[Car93]{Carter1993FiGrLieTy}
Roger~W. Carter.
\newblock {\em Finite groups of {L}ie type}.
\newblock Wiley Classics Library. John Wiley \& Sons Ltd., Chichester, 1993.
\newblock Conjugacy classes and complex characters, Reprint of the 1985
  original, A Wiley-Interscience Publication.

\bibitem[CF25]{Chen_Feng_2025_classnumber_InvariantCharacters}
Zhe Chen and Yongqi Feng.
\newblock Class numbers and invariant characters of
  $\mathfrak{sl}_2(\mathbb{F}_p)$.
\newblock {\em arXiv preprint arXiv:2508.17214}, 2025.

\bibitem[Cha24]{Chan_2024_Scalar_Product}
Charlotte Chan.
\newblock The scalar product formula for parahoric {D}eligne--{L}usztig
  induction.
\newblock {\em arXiv preprint arXiv:2405.00671}, 2024.

\bibitem[Che20]{Chen_2019_flag_orbit}
Zhe Chen.
\newblock Flags and orbits of connected reductive groups over local rings.
\newblock {\em Math. Ann.}, 376(3-4):1449--1466, 2020.

\bibitem[Che21]{Chen_2016_GenericCharSh}
Zhe Chen.
\newblock Generic character sheaves on reductive groups over a finite ring.
\newblock {\em J. Pure Appl. Algebra}, 225(3):Paper No. 106521, 14, 2021.

\bibitem[Che24]{Chen_2024_Hopfalg_duality}
Zhe Chen.
\newblock Hopf algebra and the duality operation for
  $\mathfrak{gl}_n(\mathbb{F}_q)$.
\newblock {\em arXiv preprint arXiv:2408.06191}, 2024.

\bibitem[Che25]{Chen_2025_stability_higher_Coxeter_unipotent}
Zhe Chen.
\newblock On a stability of higher level {C}oxeter unipotent representations.
\newblock {\em manuscripta math.}, 176(5), 2025.

\bibitem[CMO17]{Crisp_Meir_Onn_varHarishChandra}
Tyrone Crisp, Ehud Meir, and Uri Onn.
\newblock A variant of {H}arish-{C}handra functors.
\newblock {\em Journal of the {I}nstitute of {M}athematics of {J}ussieu}, pages
  1--57, 2017.

\bibitem[CMO24]{Crisp_Meir_Onn_inductive_approach_GLn}
Tyrone Crisp, Ehud Meir, and Uri Onn.
\newblock An inductive approach to representations of general linear groups
  over compact discrete valuation rings.
\newblock {\em Advances in Mathematics}, 440, 2024.

\bibitem[CO25]{Chan_Oi_2025_GreenFunc}
Charlotte Chan and Masao Oi.
\newblock Green functions for positive-depth {D}eligne--{L}usztig induction.
\newblock {\em arXiv preprint arXiv:2506.04449}, 2025.

\bibitem[CS17]{ChenStasinski_2016_algebraisation}
Zhe Chen and Alexander Stasinski.
\newblock The algebraisation of higher {D}eligne--{L}usztig representations.
\newblock {\em Selecta Math. (N.S.)}, 23(4):2907--2926, 2017.

\bibitem[CS26]{ChenStasinski_2023_algebraisation_II}
Zhe Chen and Alexander Stasinski.
\newblock The algebraisation of higher level {D}eligne--{L}usztig
  representations {II}: odd levels.
\newblock {\em Advances in Mathematics}, 498:111031, 2026.

\bibitem[DL76]{DL1976}
Pierre Deligne and George Lusztig.
\newblock Representations of reductive groups over finite fields.
\newblock {\em Ann. of Math. (2)}, 103(1):103--161, 1976.

\bibitem[DM20]{DM_book_2nd_edition}
Fran\c{c}ois Digne and Jean Michel.
\newblock {\em Representations of finite groups of {L}ie type}, volume~95 of
  {\em London Mathematical Society Student Texts}.
\newblock Cambridge University Press, Cambridge, second edition, 2020.

\bibitem[FKS23]{Fintzen_Kaletha_Spice_twistedYu_DMJ_2023}
Jessica Fintzen, Tasho Kaletha, and Loren Spice.
\newblock A twisted {Yu} construction, {Harish-Chandra} characters, and
  endoscopy.
\newblock {\em Duke Mathematical Journal}, 172(12):2241--2301, 2023.

\bibitem[G{\'e}r73]{Gerardin1973GL_n}
Paul G{\'e}rardin.
\newblock Sur les repr\'esentations du groupe lin\'eaire g\'en\'eral sur un
  corps {$p$}-adique.
\newblock In {\em S\'eminaire {D}elange-{P}isot-{P}oitou (14e ann\'ee:
  1972/73), {T}h\'eorie des nombres, {F}asc. 1, {E}xp. {N}o. 12}, page~24.
  Secr\'etariat Math\'ematique, Paris, 1973.

\bibitem[G{\'e}r77]{Gerardin_1977_Weil}
Paul G{\'e}rardin.
\newblock Weil representations associated with finite fields.
\newblock {\em Journal of Algebra}, 46:54--101, 1977.

\bibitem[GM20]{Geck_Malle_2020book}
Meinolf Geck and Gunter Malle.
\newblock {\em The character theory of finite groups of {L}ie type}, volume 187
  of {\em Cambridge Studies in Advanced Mathematics}.
\newblock Cambridge University Press, Cambridge, 2020.

\bibitem[Gre55]{Green_1955}
James~A. Green.
\newblock The characters of the finite general linear groups.
\newblock {\em Trans. Amer. Math. Soc.}, 80:402--447, 1955.

\bibitem[Gre61]{Greenberg19611}
Marvin~J. Greenberg.
\newblock Schemata over local rings.
\newblock {\em Ann. of Math. (2)}, 73:624--648, 1961.

\bibitem[Gre63]{Greenberg19632}
Marvin~J. Greenberg.
\newblock Schemata over local rings. {II}.
\newblock {\em Ann. of Math. (2)}, 78:256--266, 1963.

\bibitem[Hil93]{Hill_1993_Jordan}
Gregory Hill.
\newblock A {J}ordan decomposition of representations for
  {$\mathrm{GL}_n(\mathcal{O})$}.
\newblock {\em Comm. Algebra}, 21(10):3529--3543, 1993.

\bibitem[Hil94]{Hill_1994_nilpotent}
Gregory Hill.
\newblock On the nilpotent representations of
  {${\mathrm{GL}}_n({\mathcal{O}})$}.
\newblock {\em manuscripta math.}, 82(3-4):293--311, 1994.

\bibitem[Hil95a]{Hill_1995_Regular}
Gregory Hill.
\newblock Regular elements and regular characters of
  {${\mathrm{GL}}_n({\mathcal{O}})$}.
\newblock {\em J. Algebra}, 174(2):610--635, 1995.

\bibitem[Hil95b]{Hill_1995_semisimple}
Gregory Hill.
\newblock Semisimple and cuspidal characters of
  {${\mathrm{GL}}_n({\mathcal{O}})$}.
\newblock {\em Comm. Algebra}, 23(1):7--25, 1995.

\bibitem[HS19]{Hasa_Stasinski_2019_trans_AMS}
Jokke H\"as\"a and Alexander Stasinski.
\newblock Representation growth of compact linear groups.
\newblock {\em Trans. Amer. Math. Soc.}, 372(2):925--980, 2019.

\bibitem[HS22]{Hassain_Singla_2022_ADV}
M.~Hassain and Pooja Singla.
\newblock Representation growth of compact special linear groups of degree two.
\newblock {\em Adv. Math.}, 396:Paper No. 108164, 61, 2022.

\bibitem[Hum78]{Humphreys_intro_Lie_RepThy}
James~E. Humphreys.
\newblock {\em Introduction to {L}ie algebras and representation theory},
  volume~9 of {\em Graduate Texts in Mathematics}.
\newblock Springer-Verlag, New York-Berlin, 1978.
\newblock Second printing, revised.

\bibitem[Isa06]{Isaacs_CharThy_Book}
I.~Martin Isaacs.
\newblock {\em Character theory of finite groups}.
\newblock AMS Chelsea Publishing, Providence, RI, 2006.

\bibitem[JW24]{Jing_Wu_2024_J_Alg}
Naihuan Jing and Yu~Wu.
\newblock Characters of $\mathrm{GL}_n(\mathbb{F}_q)$ and vertex operators.
\newblock {\em Journal of Algebra}, 653:109--132, 2024.

\bibitem[KOS18]{Krakovski_Onn_Singla_regularchar_2018}
Roi Krakovski, Uri Onn, and Pooja Singla.
\newblock Regular characters of groups of type {$A_n$} over discrete valuation
  rings.
\newblock {\em J. Algebra}, 496:116--137, 2018.

\bibitem[Kut73]{kutzko1973characters}
Philip~C. Kutzko.
\newblock The characters of the binary modular congruence group.
\newblock {\em Bull. Amer. Math. Soc.}, 79(4):702--704, 1973.

\bibitem[LS77]{Lusztig_Srinivasan_char_finite_unitary_gp_1977}
George Lusztig and Bhama Srinivasan.
\newblock The characters of the finite unitary groups.
\newblock {\em J. Algebra}, 49(1):167--171, 1977.

\bibitem[Lus76]{Lusztig_1976_finiteness_unipotent_classes}
G.~Lusztig.
\newblock On the finiteness of the number of unipotent classes.
\newblock {\em Invent. Math.}, 34(3):201--213, 1976.

\bibitem[Lus84]{Lusztig_84_OrangeBook}
George Lusztig.
\newblock {\em Characters of reductive groups over a finite field}, volume 107
  of {\em Annals of Mathematics Studies}.
\newblock Princeton University Press, Princeton, NJ, 1984.

\bibitem[Lus88]{Lusztig_1988_Rep_red_gp_disconnectedcentre}
George Lusztig.
\newblock On the representations of reductive groups with disconnected centre.
\newblock In {\em {O}rbites unipotentes et repr\'esentations - I. {G}roupes
  finis et {A}lg\`ebres de {H}ecke}, number 168 in {A}st\'erisque, pages
  157--166. {S}oci\'et\'e math\'ematique de {F}rance, 1988.

\bibitem[Lus04]{Lusztig2004RepsFinRings}
George Lusztig.
\newblock Representations of reductive groups over finite rings.
\newblock {\em Represent. Theory}, 8:1--14, 2004.

\bibitem[Mon23]{monteiro2023stable}
Nariel Monteiro.
\newblock The stable representations of $\mathrm{GL}_{N}$ over finite local
  principal ideal rings.
\newblock {\em Communications in Algebra}, 51(6):2628--2643, 2023.

\bibitem[Nag76]{Nagornyj_1976_GL2Zpn}
S.~V. Nagornyj.
\newblock Complex representations of the group
  {$GL(2,\mathbb{Z}/p^n\mathbb{Z})$}.
\newblock {\em Zap. Nau\v{c}n. Sem. Leningrad. Otdel. Mat. Inst. Steklov.
  (LOMI)}, 64:95--103, 161, 1976.
\newblock English transl.: J. Soviet Math. \textbf{17} (1981), 1777--1783.

\bibitem[Nag78]{Nagornyi_1978_GL3}
S.~V. Nagornyj.
\newblock Complex representations of the general linear group of degree three
  modulo a power of a prime.
\newblock {\em Zap. Nau\v{c}n. Sem. Leningrad. Otdel. Mat. Inst. Steklov.
  (LOMI)}, 75:143--150, 197--198, 1978.

\bibitem[Nav18]{Navarro_2018_book_McKayConj}
Gabriel Navarro.
\newblock {\em Character theory and the {M}c{K}ay conjecture}, volume 175 of
  {\em Cambridge Studies in Advanced Mathematics}.
\newblock Cambridge University Press, Cambridge, 2018.

\bibitem[Nob77]{Nobs_1977_GL2}
Alexandre Nobs.
\newblock {Die irreduziblen Darstellungen von $GL_2(\mathbb{Z}_p)$,
  insbesondere $GL_2(\mathbb{Z}_2)$}.
\newblock {\em Math. Ann.}, 229(2):113--133, 1977.

\bibitem[NW74]{Nobs_Wolfart_theta_I_1974}
Alexandre Nobs and J\"urgen Wolfart.
\newblock Darstellungen von
  {${\mathrm{SL}}(2,{\mathbf{Z}}/p^{\lambda}{\mathbf{Z}})$} und
  {T}hetafunktionen. {I}.
\newblock {\em Math. Z.}, 138:239--254, 1974.

\bibitem[Onn08]{Onn_AdvMath_2008}
Uri Onn.
\newblock Representations of automorphism groups of finite
  {$\mathfrak{o}$}-modules of rank two.
\newblock {\em Adv. Math.}, 219(6):2058--2085, 2008.

\bibitem[OPS25]{OnnPrasadSingla_2025_zetaA2poschar}
Uri Onn, Amritanshu Prasad, and Pooja Singla.
\newblock Representation zeta functions of groups of type {$A_2$} in positive
  characteristic.
\newblock {\em Int. Math. Res. Not. IMRN}, 2025(2):Paper No. rnae270, 2025.

\bibitem[SS17]{Stasinski_Stevens_2016_regularRep}
Alexander Stasinski and Shaun Stevens.
\newblock The regular representations of {${\mathrm GL}_N$} over finite local
  principal ideal rings.
\newblock {\em Bull. Lond. Math. Soc.}, 49(6):1066--1084, 2017.

\bibitem[Sta09a]{Sta2009smooth}
Alexander Stasinski.
\newblock The smooth representations of {${\mathrm{GL}}_2(\mathfrak{o})$}.
\newblock {\em Comm. Algebra}, 37(12):4416--4430, 2009.

\bibitem[Sta09b]{Sta2009Unramified}
Alexander Stasinski.
\newblock Unramified representations of reductive groups over finite rings.
\newblock {\em Represent. Theory}, 13:636--656, 2009.

\bibitem[Sta11]{Sta2011ExtendedDL}
Alexander Stasinski.
\newblock Extended {D}eligne-{L}usztig varieties for general and special linear
  groups.
\newblock {\em Adv. Math.}, 226(3):2825--2853, 2011.

\bibitem[Sta17]{Stasinski_2016_survey}
Alexander Stasinski.
\newblock Representations of {$\mathrm{GL}_N$} over finite local principal
  ideal rings: an overview.
\newblock In {\em Around {L}anglands correspondences}, volume 691 of {\em
  Contemp. Math.}, pages 337--358. Amer. Math. Soc., Providence, RI, 2017.

\bibitem[Wol75]{Wolfart_theta_II_1975}
J\"urgen Wolfart.
\newblock Darstellungen von
  {${\mathrm{SL}}(2,{\mathbf{Z}}/p^{\lambda}{\mathbf{Z}})$} und
  {T}hetafunktionen. {II}.
\newblock {\em Manuscripta Math.}, 17(4):339--362, 1975.

\bibitem[Yu01]{Yu_2001_JAMS}
Jiu-Kang Yu.
\newblock Construction of tame supercuspidal representations.
\newblock {\em J. Amer. Math. Soc.}, 14(3):579--622, 2001.

\bibitem[Zel81]{Zelevinsky_1981_bk}
Andrey~V. Zelevinsky.
\newblock {\em Representations of finite classical groups}, volume 869 of {\em
  Lecture Notes in Mathematics}.
\newblock Springer-Verlag, Berlin-New York, 1981.
\newblock A Hopf algebra approach.

\end{thebibliography}

\end{document}